\documentclass[12pt,oneside,a4paper]{amsart}
\usepackage[british]{babel}

\usepackage{upgreek}
\usepackage{graphicx}
\usepackage{amscd}
\usepackage{amsmath}
\usepackage{amsfonts}
\usepackage{amssymb}
\usepackage{mathrsfs}
\usepackage{comment}
\excludecomment{mycomment}
\usepackage{color}
\usepackage[usenames,dvipsnames,table,xcdraw]{xcolor}
\usepackage{pdfsync}
\usepackage{bbm}
\usepackage{dsfont}
\usepackage[colorlinks=true]{hyperref}
\usepackage{enumerate}
\usepackage{booktabs}
\usepackage{float}
\usepackage{wrapfig}
\usepackage{caption}
\usepackage{subcaption}
\usepackage{longtable}
\usepackage{listings}
\usepackage[T1]{fontenc}
\usepackage[scaled]{beramono}
\usepackage{tikz}
\usepackage{pgffor}
\usepackage[all]{xy}
\usepackage{bbold}
\usepackage{enumerate}
\usepackage{enumitem}
\definecolor{trp}{rgb}{1,1,1}

\definecolor{red}{rgb}{1,0,.2}

\usepackage{amsthm}
\theoremstyle{plain}
\newtheorem{theorem}{Theorem}[section]

\newtheorem*{acknowledgement}{Acknowledgments}

\newtheorem{fact}[theorem]{Fact}

\newtheorem{corollary}[theorem]{Corollary}

\newtheorem{definition}[theorem]{Definition}

\newtheorem{lemma}[theorem]{Lemma}

\newtheorem{proposition}[theorem]{Proposition}

\numberwithin{equation}{section}

\newcommand{\N}{\mathbb{N}}

\newcommand{\ii}{\mathbf{i}}

\newcommand{\iiv}{\overline{\imath}}
\newcommand{\jjv}{\overline{\jmath}}

\newcommand*{\arabicdec}[1]{\the\numexpr\value{#1}\relax}

\usepackage{anysize}
\papersize{24.5cm}{16.2006cm}

\marginsize{1.5cm}{1cm}{0cm}{0cm}

\usepackage{caption}

\definecolor{blue}{rgb}{0,0,1}

\definecolor{red}{rgb}{1,0,.7}

\usepackage{tikz}
\usetikzlibrary{cd,decorations.pathreplacing,positioning,arrows}

\begin{document}
\title[CPLIFS]{Piecewise linear iterated function systems on the line of overlapping construction}

\author{R. D\'aniel Prokaj}
\address{ R. D\'aniel Prokaj, Budapest University of Technology and Economics, Hungary} \email{prokajrd@math.bme.hu}

\author{K\'aroly Simon}
\address{K\'aroly Simon, Budapest University of Technology and Economics, MTA-BME Stochastics Research Group, P.O. Box 91, 1521 Budapest, Hungary} \email{simonk@math.bme.hu}

\begin{abstract}
In this paper we consider Iterated Function Systems (IFS) on the real line consisting of continuous piecewise linear functions.
We assume some bounds on the contraction ratios of the functions, but we do not assume any separation condition. Moreover, we 
do not require that the functions of the IFS are injective, but we assume that their derivatives are separated from zero.
 We prove that if we fix all the slopes but perturb all other parameters, then for all parameters outside of an exceptional set of less than full packing dimension, the Hausdorff dimension of the attractor is equal to the exponent which comes from the most natural system of covers of the attractor.
\end{abstract}
\date{\today}

\maketitle

\tableofcontents

\thispagestyle{empty}
\section{Introduction}

\subsection{Attractors on the line}
Iterated Function Systems (IFS) on the line  consist of finitely many strictly contracting self-mappings of $\mathbb{R}$. It was proved by Hutchinson \cite{hutchinson1981fractals}
that
for every IFS $\mathcal{F}=\left\{f_k\right\}_{k=1}^{m}$ there is a unique non-empty compact set $\Lambda^{\mathcal{F}}$ which is
called the attractor of the IFS $\mathcal{F}$ and defined by
 \begin{equation}\label{cr65}
   \Lambda^{\mathcal{F}}=\bigcup\limits_{k=1}^{m}f_k(\Lambda^{\mathcal{F}}).
 \end{equation}

 It is easy to see that for every IFS $\mathcal{F}$ there exists a unique "smallest" non-empty compact interval $I^{\mathcal{F}}$  which is sent into itself by all the mappings of $\mathcal{F}$:
 \begin{equation}
 \label{cr22}
I^{\mathcal{F}}:=\bigcap\left\{ J:
J\subset \mathbb{R}\mbox{ compact interval with }
f_k(J)\subset J \mbox{, for all }k\in[m]
 \right\},
 \end{equation}
 where $[m]:=\left\{ 1,\dots  ,m \right\}$. 
It is easy to see that 
\begin{equation}
\label{cr15}
\Lambda^{\mathcal{F}}=\bigcap\limits_{n=1}^{\infty}
\bigcup\limits_{(i_1,\dots  ,i_n)\in[m]^n}
I_{i_1\dots  i_n}^{\mathcal{F}} ,
\end{equation}
where $I_{i_1\dots  i_n}^{\mathcal{F}}:=
f_{i_1\dots  i_n}(I^{\mathcal{F}})$ are the \texttt{cylinder intervals}, and we use the common shorthand notation $f_{i_1\dots  i_n}:=f_{i_1}\circ\cdots \circ f_{i_n}$ for an
$(i_1,\dots  ,i_n)\in [m]^n$.

The following three IFS families, each defined on $\mathbb{R}$, appear on Figure \ref{cr66} . 

\begin{enumerate}
[label={\bf (a)}]
    \item Self-similar IFS:  
$\mathcal{F}=\left\{f_i(x)=\rho_ix+t_i\right\}_{k=1}^{m}$, where $\rho_i\in(-1,1)\setminus\left\{0\right\}$ and $t_i\in\mathbb{R}$.
    \item Hyperbolic IFS: $\mathcal{F}=\left\{ f_1,\dots  ,f_m \right\}$,  where each $f_k:J\to J$ is a $C^{1+\varepsilon }(J)$ contracting self-mappings of a non-empty open interval $J\subset \mathbb{R}$.
    \item We introduce the \texttt{Continuous Piecewise Linear Iterated Function Systems (CPLIFS)}: These are IFSs of the form
    $\mathcal{F}=\left\{ f_1,\dots  ,f_m \right\}$, where  
    $f_k:\mathbb{R}\to\mathbb{R}$ are continuous, piececewise linear contractions (not  necessarily injective) with all slopes different from zero. 
\end{enumerate}
\begin{figure}[H]
  \centering
  \includegraphics[width=13cm]{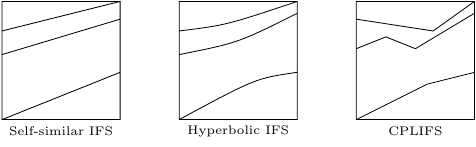}
  \caption{}\label{cr66}
\end{figure}

We note that the interval $I^{\mathcal{F}}$
introduced in \eqref{cr22} is the convex hull of the attractor in the first two cases, but not necessarily for non-injective CPLIFSs.

 The $t$-dimensional Hausdorff measure of an open Borel set $A\subset\mathbb{R}$ is  
\begin{equation}
\label{cr20}
\mathcal{H}^t(A):=\sup\limits_{\delta>0}
\left\{ 
\inf\left\{ 
\sum_{i=1}^{\infty}
|A_i|^t:
A\subset \bigcup\limits_{i=1}^{\infty} A_i,\ |A_i|<\delta 
 \right\}
 \right\}.
\end{equation}
Let $\mathcal{F}$ be an arbitrary IFS on the line.
We can give a natural upper bound on  $\mathcal{H}^t(\Lambda^{\mathcal{F}} )$ if we substitute the covering system made of the cylinder intervals $\left\{ I^{\mathcal{F}}_{i_1\dots  i_n}  \right\}_{(i_1\dots  i_n)\in[m]^n}$ in the place of the most efficient cover in \eqref{cr20},
just like in \eqref{cr15}. For this particular system the right hand side of \eqref{cr20} is related to the sequence of sums $\left\{ S_n^t \right\}_{n=1}^{\infty}$
\begin{equation}
\label{cr19}
S_n^t:=\sum_{(i_1\dots  i_n)\in[m]^n}
|I^{\mathcal{F}}_{i_1\dots  i_n}|^t.
\end{equation}
Namely, if the exponential growth rate of $\left\{ S_n^t \right\}_{n=1}^{\infty}$ is negative (positive), then it suggests that $\mathcal{H}^t(\Lambda^{\mathcal{F}})$ is equal to zero (infinity). Therefore the minimum of the dimension of the ambient space 
and the value of the exponent $t$ for which this exponential growth rate is zero is a good natural guess for the Hausdorff dimension of $\Lambda^{\mathcal{F}}$.
This motivates the introduction of the exponential growth rate  
\begin{equation}\label{cr64}
    \Phi^{\mathcal{F}}(s):=\limsup_{n\rightarrow\infty}\frac{1}{n}\log \sum_{i_1\dots i_n} |I^{\mathcal{F}}_{i_1\dots i_n}|^s.
\end{equation}
It is easy to see that we can obtain  $\Phi ^{\mathcal{F}}(s)$
above as a special case of the non-additive upper capacity topological pressure introduced by Barreira in \cite[p. 5]{barreira1996non}. 
According to \cite[Theorem 1.9]{barreira1996non}, the zero of $\Phi^{\mathcal{F}}(s)$ is well defined
\begin{equation}\label{cr61}
    s_{\mathcal{F}}:=(\Phi^{\mathcal{F}})^{-1}(0).
\end{equation}

Barreira considered generalized Moran constructions \cite[Section 2.1.2]{barreira1996non}. Condition (b) of such a construction requires that the cylinder intervals $\left\{ I^{\mathcal{F}}_{i_1\dots  i_n} \right\}$ are sufficiently well separated. This separation assumption does not hold for the constructions considered in this paper, thus not all of his results apply.
In the last inequality of 
\cite[Theorem 2.1 part (a)]{barreira1996non}
Barreira proves that the upper box dimension of $\Lambda ^{\mathcal{F}}$ is less than or equal to the root of the non-additive upper capacity topological pressure.
It is easy to see that its proof does not require any separation conditions on the cylinder intervals. That is 
\begin{corollary}[Barreira]\label{cr12}
  For any IFS $\mathcal{F}$ on the line 
  \begin{equation}
\label{cr13}
\overline{\dim}_{\rm B}  \Lambda ^{\mathcal{F}}\leq
s_{\mathcal{F}.}
\end{equation}
\end{corollary}

For a given IFS $\mathcal{F}$ on the line we name $s_{\mathcal{F}}$
the \texttt{natural dimension of the system}.
  It is easy to see that
in the self-simiar case $s_{\mathcal{F}}$ is the solution of the self-similar equation
\begin{equation}\label{cr57}
\sum_{k=1}^{m} |\rho_k|^{s_{\mathcal{F}}}=1.
\end{equation}
We call $s_{\mathcal{F}}$ the similarity dimension in this case. 
If $\mathcal{F}$ is a hyperbolic system, then $s_{\mathcal{F}}$
is the root of the so-called pressure formula (see \cite{przytycki2010conformal}). In both cases the Open Set Condition (OSC) implies that
\begin{equation}
\label{cr16}
\dim_{\rm H}  \Lambda^{\mathcal{F}} =\min\left\{ s_{\mathcal{F}},1 \right\}.
\end{equation}
The cylinder intervals of the attractor of a system that satisfies the OSC are well separated 
(see \cite[p. 35]{falconer1997techniques}).
However, less strict separation conditions might lead to the same result.
In the self-similar case the celebrated Hochman Theorem \cite[Theorem 1.1]{hochman2014self} yields that 
\eqref{cr16} also follows from the Exponential Searation Condition 
(ESC).

\subsection{Introducing Continuous Piecewise Linear IFSs}

The IFSs we consider in this paper are consisting of piecewise linear functions. Thus their derivatives might change at some points, but they are linear over given intervals of $\mathbb{R}$. We always assume that the functions are continuous, piecewise linear, strongly contracting with non-zero slopes, and that the slopes can only change at finitely many points. This setup enables us to investigate the case of non-injective functions as well, which is not achievable with either hyperbolic or self-similar systems.

A special case of CPLIFSs can be handled using the theory of F. Hofbauer \cite{hofbauer1996box} and P. Raith \cite{raith1994continuity}.
In particular, a family of CPLIFSs satisfying the one dimensional version of the rectangular open set condition.
\begin{definition}
  We say that a CPLIFS $\mathcal{F}=\{ f_k\}_{k=1}^m$ satisfies the \texttt{Interval Open Set Condition} (IOSC) if the first cylinder intervals 
  $\{I^{\mathcal{F}}_k\}_{k=1}^m$ are pairwise disjoint.
\end{definition}
Take a CPLIFS $\mathcal{F}=\{ f_k\}_{k=1}^m$ of injective functions that satisfies the IOSC. Each $f_k$ function can be considered as the local inverse over $I^{\mathcal{F}}_k$ of a strictly expanding map on $\mathbb{R}$. Therefore, the Hausdorff dimension of the attractor of such systems is equal to the root of a topological pressure function \cite{raith1994continuity}. 
Further, according to \cite{hofbauer1996box}, for this special CPLIFS family the Hausdorff and the box dimensions of the attractor are also equal.

In general, when we have no information about the possibble overlapping of the cylinders, we can claim only that \eqref{cr16} holds in some sense typically. Instead of a particular IFS, it is natural to consider its so-called translation family.
 \begin{definition}\label{cv80}
  For every $\pmb{\tau}=(\tau_1, \dots ,\tau_m)\in\mathbb{R}^m$
  we define the translation family of $\mathcal{F}$ by
  \begin{equation}\label{cv79}
   \left\{ \mathcal{F}^{\pmb{\tau}}\right\}_{\pmb{\tau}\in\mathbb{R}^m}, \mbox{ where }
    \mathcal{F}^{\tau}=\{f_1^{\pmb{\tau}}, \dots ,f_m^{\pmb{\tau}}\}
  \end{equation}
  and $f_k^{\pmb{\tau}}(x):=f_k(x)+\tau_k$ for a $k\in[m]$. Moreover,
for an $\mathbf{i}=(i_1, \dots ,i_n)\in\Sigma^*$ we define the function
  \begin{equation}\label{cv75}
    f_{\mathbf{i}}^{\pmb{\tau}}(x):=f_{i_1}^{\pmb{\tau}}\circ\cdots\circ
    f_{i_n}^{\pmb{\tau}}(x).
  \end{equation}
\end{definition}

The classical results \cite{falconer1987hausdorff},
  \cite{falconer1988hausdorff} and \cite{simon2001invariant}
  about the translation families read like this: For typical (in some sense) translations $\pmb{\tau}$,
 the Hausdorff dimension of the translated attractor $\Lambda^{\pmb{\tau}}$ is equal to its natural dimension. Here the sense of typicality depends on the family of the IFS considered. For example, for hyperbolic IFS on the line the typicality above means "typical with respect to the $m$-dimensional Lebesgue measure" \cite{simon2001invariant} (at least in the case when all contractions are stronger than $\frac12$). In this case we say that the Hausdorff dimension and the natural dimension coincide for Lebesgue typical translations.
 On the other hand, Hochman \cite{Hochman_2015} proved a much stronger theorem for the translation family of self-similar IFSs that we call \texttt{Multi Parameter Hochman Theorem} (see Theorem \ref{cv78}). This theorem implies that the
 exceptional set of translations $\pmb{\tau}$ for which
 the ESC does not hold for $\mathcal{F}^{\pmb{\tau}}$ (and consequently
 $\dim_{\rm H} \Lambda^{\pmb{\tau}}$ is different from the similarity dimension), has packing dimension (and in this way also Hausdorff dimension) less than or equal to $m-1$. Remember that the dimension of the parameter space was equal to $m$.
 Motivated by this, we introduce the following terminology related to more general IFS families on the line:

\medskip
 \underline{\textbf{Terminology}:}
 Let $\left\{\mathcal{F}^{\pmb{\lambda}}\right\}_{\pmb{\lambda}\in U}$ be a family of IFSs on the line where the parameter set $U$ is an open
subset of $\mathbb{R}^d$ for some $d\geq 1$.
 We say that a \texttt{property} $\mathfrak{P}$, which makes sense for all elements of this family, \texttt{holds  $\dim_{\rm P} $-typically} if the exceptional set $E$ of those parameters $\pmb{\lambda}\in U$
for which $\mathfrak{P}$ does not hold  satisfies $\dim_{\rm P} E<d$.
That is
the packing dimension of the exceptional set
 $E \subset U$ is smaller than the dimension of the parameter space $U$.
\medskip

Our main result states that $\dim_{\rm P} $-typically the Hausdorff dimension of the attractor of a CPLIFS that satisfies certain regularity conditions is equal to its natural dimension. 
More precisely, 
a CPLIFS $\mathcal{F}$ is called regular if it has sufficiently small contraction ratios and there exists an $N$ such that for all $(i_1,\dots  ,i_N)\in[m]^N$ 
the cylinder interval $I_{i_1\dots  i_N}^{\mathcal{F}}$ 
does not contain any point of non-differentiability of any of the functions of $\mathcal{F}$.
We will show the following:
\begin{enumerate}[label={\bf (a)}]
\item For a regular CPLIFS $\mathcal{F}$ we consider its generated self-similar IFS $\mathcal{S}_{\mathcal{F}}$ (see Section \ref{ct16}). If it satisfies the so-called Exponential Separation Condition (ESC), then its $N$-th iterate system $\mathcal{S}_{\mathcal{F}}^N$ and any subsystem of 
$\mathcal{S}_{\mathcal{F}}^N$ will also satisfy the ESC.
 We point out that we can select a suitable subsystem $\mathscr{S}_{\mathcal{F}}\subset \mathcal{S}_{\mathcal{F}}^N$ and form a graph-directed self-similar IFS from the functions of 
$\mathscr{S}_{\mathcal{F}}$
such that the attractor of this graph-directed system and $\Lambda ^{\mathcal{F}}$ coincide. 
Then we use the Jordan Rapaport Theorem
\cite{jordan2020dimension}
to compute the dimension of this graph-directed attractor which, as we just mentioned, is the same as 
$\Lambda ^{\mathcal{F}}$. In this way we obtain that 
    \begin{equation}\label{cr26}
      \dim_{\rm H} \Lambda^{\mathcal{F}}=s_{\mathcal{F}}.
    \end{equation}
\item We verify in our Main Proposition that for a $\dim_{\rm P} $-typical set of parameters (see our Terminology)
the above mentioned regularity property holds.
Then we apply the multi-parameter Hochman Theorem \cite[Theorem 1.10]{Hochman_2015} to conclude that the ESC is valid also for a $\dim_{\rm P} $-typical set of parameters.
That is by part (a) \eqref{cr26} holds $\dim_{\rm P} $-typically.
\end{enumerate}
To state these conditions more precisely we need some notations.
See Figure \ref{cv47} for guidance.

\begin{figure}[t]
\centering
\includegraphics[width=13cm]{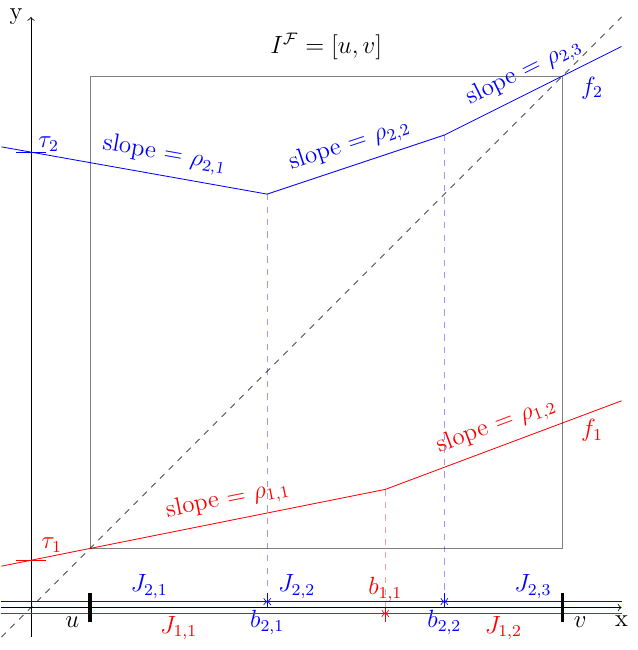}
\caption{A general CPLIFS with the related notations.}\label{cv47}
\end{figure}

\subsection{Notations}\label{cr55}
Let $\mathcal{F}=\left\{f_k\right\}_{k=1}^{m}$ be a CPLIFS and 
$I^{\mathcal{F}} \subset \mathbb{R}$ be the compact interval defined in \eqref{cr22}. For any $k\in [m]$
let $l(k)$ be the number of breaking points $\left\{b_{k,i}\right\}_{i=1}^{l(k)}$ of $f_k$.
They determine the $l(k)+1$ open intervals of linearity
 $\left\{J_{k,i}\right\}_{i=1}^{l(k)+1}$ (see Figure \ref{cv47}). We write $S_{k,i}$ for the contracting similarity on $\mathbb{R}$ that satisfies
 $S_{k,i}|_{J_{k,i}}\equiv f_k|_{J_{k,i}}$.
 We define $\left\{\rho_{k,i}\right\}_{k\in[m],i\in [l(k)+1]}$ and $\left\{t_{k,i}\right\}_{k\in[m],i\in [l(k)+1]}$  such that
\begin{equation}\label{cr56}
  S_{k,i}(x)=\rho_{k,i}x+t_{k,i}.
\end{equation}
 We say that $\mathcal{S}_{\mathcal{F}}:=\left\{S_{k,i}\right\}_{k\in[m],i\in [l(k)+1]}$ is the \texttt{self-similar IFS  generated by the CPLIFS  $\mathcal{F}$}. In order to prove a certain transversality like condition we need to require that the system is strongly contracting.
 Namely, let
 \begin{equation}\label{cr52}
   \rho_k:=\max_{i\in [l(k)]}{|\rho_{k,i}|},\quad \rho_{\max}:=\max_k \rho_k, \mbox{ and }
   \rho_{\min}:=\min_{k\in [m]}\min_{i\in [l(k)]} |\rho_{k,i}|
 \end{equation}
 \begin{definition}\label{cr54}
   We say that \texttt{$\mathcal{F}$ is small} if both of the following two requirements hold:
  \begin{enumerate}[label={\bf (a)}]
  \item $\sum\limits_{k=1}^{m}\rho_k<1$.
  \item  Our second requirement depends on the injectivity of $f_k$:
  \begin{enumerate}[label={\bf (i)}]
  \item If $f_k$ is injective then we require that $\rho_k<\frac{1}{2}$.
  \item If $f_k$ is not injective then we require that $\rho_k<\frac{1-\rho_{\max}}{2}$, which always holds if $\rho_{\max}<\frac{1}{3}$.
  \end{enumerate}
  \end{enumerate}
   \end{definition}
The first assumption is required to know a priori that $\dim_{\rm H} \Lambda < 1$ in a certain uniform manner, and the later assumptions are required for a kind of transversality argument.
   \begin{definition}\label{cr53}
     We say that a small CPLIFS $\mathcal{F}$ is \texttt{regular} if its attractor $\Lambda_{\mathcal{F}}$ does not contain any of the breaking points $\left\{b_{k,i}\right\}_{k\in[m],
     i\in[l(k)]}$.
   \end{definition}

\medskip
Remember, our aim is to prove that \emph{$\dim_{\rm P}$-typically}
\begin{equation}\label{cs69}
\mbox{ $\mathcal{F}$ is small}
  \Longrightarrow
  \mbox{$\mathcal{F}$ is regular,} \quad\&\quad
  \mbox{$\mathcal{F}$ is regular}  \Longrightarrow
  \dim_{\rm H} \Lambda = s_{\mathcal{F}}.
\end{equation}
Here \emph{$\dim_{\rm P}$-typically} means that if we fix all the contraction ratios satisfying the condition in the left side
but perturb all other parameters, then for all parameters other then an exceptional set of less than full packing dimension the statement on the right in \eqref{cs69} holds, for each implication respectively.

The translation family of CPLIFSs does not have enough parameters to claim a $\dim_{\rm P} $-typical result for the Hausdorff dimension of the attractor.
Thus, besides the vertical translations (the translations of the graphs of the functions in the IFS vertically), as in the translation families, in the case of CPLIFSs we also consider the breaking points as parameters (but we fix all of the contraction ratios). In this way besides the  $m$ vertical translation parameters $\pmb{\tau}=(\tau_1, \dots ,\tau_m)$ we have $L:=\sum\limits_{k=1}^{m}l(k)$ parameters which correspond to all the breaking points. Let us denote temporarily this $L+m$-dimensional parameter space by $U$. For fixed contraction ratios $\left\{\rho_{k,i}\right\}_{k\in[m],i\in [l(k)+1]}$ the elements of this parameter space $U$ determine a CPLIFS.he

\subsection{The steps of our argument}
 We fix the set of contractions \newline
 $\left\{\rho_{k,i}\right\}_{k\in[m],i\in [l(k)+1]}$ which are small in the sense introduced above.
  \begin{enumerate}[label={\bf (a)}]
\item First we observe that there exists a non-singular affine correspondence between the elements of $U$ and
       $\left\{t_{k,i}\right\}_{k\in[m],i\in[l(k)+1]}$,
the set of translations of the generated self-similar system.
\item We may apply the Multi Parameter Hochman Theorem  to conclude that $\dim_{\rm P} $-typically the ESC holds for the generated self-similar IFS.
\item We prove in our Main Proposition that $\dim_{\rm p} $-typically
there are no breaking points on the attractors.
\item It implies, that $\dim_{\rm P} $-typically we can construct a self-similar graph-directed IFS whose attractor coincides with the attractor of our CPLIFS.
\item When we are in both of the $\dim_{\rm P} $-typical situations described by the previous points,
 we introduce an ergodic (Markov) measure $\nu$ on the symbolic space determined by the previously mentioned graph-directed IFS such that $\dim_{\rm H} \nu$ is equal to the Hausdorff dimension of the attractor of our CPLIFS.
\item To show it, we use a recent theorem of Jordan-Rapaport 
\cite[Theorem 1.1]{jordan2020dimension} to compute $\dim_{\rm H} \nu$. Then we conclude that the Hausdorff dimension and the natural dimension of the attractor coincide.
  \end{enumerate}

\subsection{The organization of the paper}
The first two sections are dedicated to introduce and motivate the continuous piecewise linear iterated function systems, and the pressure function we use to calculate the natural dimension of the attractor. We introduce here all the important definitions and notations, and the typicality we use in our theorems as well. At the end of the chapter we state our main theorems.

In Section \ref{cv13} we present some recent results from the field of self-similar iterated function systems, that will be useful in our proofs. Following \cite{hochman2014self} we define the exponential Separation Condition (ESC) and state some important theorems of M. Hochman that utilizes this assumption. We also state here a recent result of Jordan and Rapaport \cite{jordan2020dimension} that will help us give a lower bound on the Hausdorff dimension of the attractor of a regular CPLIFS.
In the second half of this section we introduce the self-similar graph-directed iterated function systems and their dimension theory based on \cite{mauldin1988hausdorff} and \cite{falconer1997techniques}.

To prove our results, we need to associate regular CPLIFSs to graph-directed iterated function systems. Section \ref{cs64} contains the steps how we relate these two IFS families to each other. We also introduce here an ergodic invariant measure supported on the attractor of a regular CPLIFS, which lets us calculate the Hausdorff dimension of the attractor of the corresponding CPLIFS with the help of the Jordan-Rapaport theorem.

Section \ref{cr60} contains the proof of our main theorems assuming the main proposition, and some useful lemmas that are not just helpful in the proofs, but they also give us more details about CPLIFSs.
The proof of our main proposition is contained in section \ref{cr58}.

In the Appendix we present how our method of associating a regular CPLIFS to a graph-directed function system can be used to extend one of Hochman's theorems (Namely part (b) of Theorem \ref{cr59} in this paper) from self-similar sets to graph directed ones.

\section{Preliminaries and main results}\label{cv23}

\subsection{Parametrizing CPLIFSs} We fix a number $m\geq 2$, and use it as the number of functions in a CPLIFS throughout the paper. 
Let
$\mathcal{F}=\left\{f_k\right\}_{k=1}^{m }$ be a CPLIFS.
We write $l(k)$ for the number of breaking points  of $f_k$
for $k\in[m]$,
and we say that the \texttt{type of the CPLIFS} is the vector
  \begin{equation}\label{cv44}
  \pmb{\ell }=(l(1), \dots ,l(m)).
\end{equation}
For example the type of the CPLIFS on Figure \ref{cv47} is
$\pmb{\ell }=(2,1)$. If $\mathcal{F}$ is a CPLIFS of type $\pmb{\ell }$,  then we write
$$\mathcal{F}\in\mathrm{CPLIFS}_{\pmb{\ell }}.$$
 The breaking points of $f_k$
are denoted by $b_{k,1} < \cdots < b_{k,l(k)}$.
Let $L:=\sum\limits_{k=1}^{m}l(k)$ be the \texttt{total number of breaking points} of the functions of $\mathcal{F}$ with multiplicity if some of the breaking points of two different elements of $\mathcal{F} $ coincide.
We arrange all the breaking points in an
$L$ dimensional vector $\mathfrak{b}\in \mathbb{R}^L$
 in a way described below. First we partition $[L]=\left\{1, \dots ,L\right\}$ into blocks of length $l(k)$ for $k\in[m]$. The $k$-th block is
 \begin{equation}\label{cv52}
 L^k:=\left\{p\in\mathbb{N} :1+\sum\limits_{j=1}^{k-1}l(j) \leq p
 \leq \sum\limits_{j=1}^{k}l(j)
\right\}
 \end{equation}
 where $\sum\limits_{j=1}^{k-1}$ is meant to be $0$ when $k=1$.
 We use this convention without further mentioning it throughout the paper. The breaking points of $f_k$ occupy the components belonging to
the block $L(k)$ in increasing order. That is
\begin{equation}\label{cv58}
  \mathfrak{b}
  =
  (\underbrace{b_{1,1}, \dots ,b_{1,l(1)}}_{L^1},
  \underbrace{b_{2,1}, \dots ,b_{2,l(2)}}_{L^2},
  \dots ,
  \underbrace{b_{m,1}, \dots ,b_{m,l(m)}}_{L^m}).
\end{equation}
The set of breaking points vectors $\mathfrak{b}$ for a type $\pmb{\ell }$ CPLIFS is
\begin{equation}\label{cv59}
  \mathfrak{B}^{\pmb{\ell }}:=
  \left\{\mathbf{x}\in\mathbb{R}^{L}
  : x_i<x_j \mbox{ if } i<j \mbox{ and } \exists k\in [m] \mbox{ with }
  i,j\in L^k
  \right\}.
\end{equation}
The $l(k)$ breaking points of the piecewise linear continuous function $f_k$ determines the $l(k)+1$  \texttt{intervals of linearity} $J_{k,i}^{\mathfrak{b}}$, among which the first and the last are actually half lines:
\begin{equation}\label{cv57}
 J_{k,i}:=  J_{k,i}^{\mathfrak{b}}:=
  \left\{
    \begin{array}{ll}
      (-\infty ,b_{k,1}), & \hbox{if $i=1$;} \\
      (b_{k,i-1},b_{k,i}), & \hbox{if $2\leq i\leq l(k)$;} \\
      (b_{k,l(k)},\infty ), & \hbox{if $i=l(k)+1$.}
    \end{array}
  \right.
\end{equation}
The derivative of $f_k$ exists on $J_{k,i}$ and is equal to the constant
\begin{equation}\label{cs68}
  \rho_{k,i}:\equiv f'_k|_{J_{k,i}}.
\end{equation}
We arrange the contraction ratios $\rho_{k,i}\in (-1,1)\setminus \{ 0\}$ into a vector $\pmb{\rho}$ in an analogous way as we arranged the breaking points into a vector  in \eqref{cv58}, but taking into account that there is one more contraction ratio for each $f_k$ than breaking points:
\begin{equation}\label{cv56}
  \pmb{\rho}:=\pmb{\rho}_{\mathcal{F}}:=
  (
  \underbrace{\rho_{1,1}, \dots ,\rho_{1,l(1)+1}}_{\widetilde{L}^1},
 \dots
 ,\underbrace{\rho_{m,1}, \dots ,\rho_{m,l(m)+1}}_{\widetilde{L}^m}
  )\in \left( (-1,1)\setminus \{ 0\}\right)^{L+m},
\end{equation}
where
\begin{equation}\label{cv51}
  \widetilde{L}^k:=
  \left\{ p \in\mathbb{N}:
  1+\sum\limits_{j=1}^{k-1}\left(1+l(j)\right) \leq p  \leq
  \sum\limits_{j=1}^{k}\left(1+l(j)\right) \right\}.
\end{equation}
We call $\pmb{\rho}$ the \texttt{vector of contractions}.
The set of all possible values of $\pmb{\rho}$ for an $\mathcal{F}\in \mathrm{CPLIFS}_{\pmb{\ell }}$
is
\begin{equation}\label{cv50}
  \mathfrak{R}^{\pmb{\ell }}:
  =
  \left\{\pmb{\rho}
  \in\left((-1,1)\setminus \{ 0\}\right)^{L+m}:
  \forall k\in[m],\
\!  \forall i,i+1\in\widetilde{L}^k,\!
  \rho_{i}\ne \rho_{i+1}
  \right\},
\end{equation}
where $\pmb{\rho}=(\rho_1, \dots ,\rho_{L+m})$.
Recall the definitions of $\rho_{k}$, $\rho_{\max}$ and
$\rho_{\min}$ from \eqref{cr52}.
Moreover, let $\rho_{k_1 \dots k_n}:=\rho_{k_1}\cdots\rho_{k_n}$.
Clearly,
\begin{equation}\label{ct67}
  \vert f'_{k_1 \dots k_n}(x)\vert \leq \rho_{k_1 \dots k_n},
  \mbox{ for all } x.
\end{equation}

Finally, we write
\begin{equation}\label{cv55}
  \tau_k:=f_k(0), \mbox{ and }
  \pmb{\tau}:=(\tau_1, \dots ,\tau_m)\in\mathbb{R}^{m}.
\end{equation}
So, the parameters that uniquely determine an $\mathcal{F}\in\mathrm{CPLIFS}_{\pmb{\ell }}$
can be organized into a vector
\begin{equation}\label{cv53}
  \pmb{\lambda}=(\mathfrak{b},\pmb{\tau},\pmb{\rho})\in
  \pmb{\Gamma}^{\pmb{\ell }}:=\mathfrak{B}^{\pmb{\ell }}\times\mathbb{R}^m\times\mathfrak{R}^{\pmb{\ell }}
   \subset \mathbb{R}^L\times\mathbb{R}^m\times\mathbb{R}^{L+m}=
   \mathbb{R}^{2L+2m}.
\end{equation}
It is easy to see that $\pmb{\lambda}$ indeed determines a CPLIFS uniquely. For instance, assuming that all the breaking points in the system are positive, we have the following formula for the functions of $\mathcal{F}$:
\begin{equation}\label{cr07}
    \quad f_k(x)=\tau_k + 
    \sum_{j=1}^{l(k)} (b_{k,j}-b_{k,j-1})\rho_{k,j}\pmb{1}\{ b_{k,j}\leq x\} +
    \sum_{j=1}^{l(k)+1} \rho_{k,j} x \pmb{1}\{ x\in J_{k,j}\} ,
\end{equation}
for any $k\in[m]$ and $x\geq 0$, where we used the notation $\forall k\in[m]: b_{k,0}=0$ to make the formula more compact.

For a $\pmb{\lambda}\in \pmb{\Gamma}^{\pmb{\ell }}$ we write $\mathcal{F}^{\pmb{\lambda}}$ for the corresponding CPLIFS, $\Lambda^{\pmb{\lambda}}$ for its attractor, and $s_{\pmb{\lambda}}$ for its natural dimension. Similarly,
for an $\mathcal{F}\in\mathrm{CPLIFS}_{\pmb{\ell }}$ we write
$\pmb{\lambda}(\mathcal{F})$ for the corresponding element of $\pmb{\Gamma}^{\pmb{\ell }}$. 

We can handle only those situations when the contraction ratios are sufficiently small.
Along the lines of Definition \ref{cr54}, we define the set of small contraction vectors for a given type $\pmb{\ell }$ as
\begin{equation}\label{cv46}
 \mathfrak{R}^{\pmb{\ell }}_{\mathrm{small}}:=
 \left\{ \pmb{\rho}\in\mathfrak{R}^{\pmb{\ell }} :
   \mathcal{F}^{(\mathfrak{b},\pmb{\tau},\pmb{\rho})}
   \mbox{ is small for all } \mathfrak{b}\in\mathfrak{B}^{\pmb{\ell}}
   \mbox{ and } \pmb{\tau}\in\mathbb{R}^m \right\}.
\end{equation}

\bigskip

With the help of these notations we can taylor the terminology of $\dim_{\rm P}$-typicality for CPLIFSs.

\underline{\textbf{Terminology}:}\label{cs67}
Let $\mathfrak{P}$ be a property that makes sense for every CPLIFS.   
For a contraction vector
$\pmb{\rho}\in\mathfrak{R}_{\mathrm{small}}^{\pmb{\pmb{\ell }} }$
 we consider the (exceptional) set  
\begin{equation}\label{cs71}
  E_{\mathfrak{P},\pmb{\ell }}^{\pmb{\rho}}=:
  \left\{
  (\mathfrak{b},\pmb{\tau})\in  \mathfrak{B}^{\pmb{\ell }}\times
  \mathbb{R}^m
  :
  \mathcal{F}^{(\mathfrak{b},\pmb{\tau},\pmb{\rho})} \mbox{ does not have property } \mathfrak{P}
  \right\}.
\end{equation}
We say
 that \texttt{property $\mathfrak{P}$ holds $\dim_{\rm P}$-typically} if for all
type  $\pmb{\ell }$  and for all $\pmb{\rho}\in\mathfrak{R}_{\mathrm{small}}^{\pmb{\ell }}$ we have
 \begin{equation}\label{cs70}
   \dim_{\rm P} E_{\mathfrak{P},\pmb{\ell }}^{\pmb{\rho}}  <
   L+m,
 \end{equation}
where $\pmb{\ell }=(l(1), \dots ,l(m))$ and $L=\sum\limits_{k=1}^{m}l(k)$ as above.

\subsection{Our main results}

With the help of all these definitions and notations we can state now the main results of this paper.

\begin{theorem}[Main Theorem]\label{ct48}
  For a $\dim_{\rm P}$-typical small CPLIFS $\mathcal{F}$ we have
\begin{equation}\label{cv09}
  \dim_{\rm H} \Lambda^{\mathcal{F}}=
  \dim_{\rm B} \Lambda^{\mathcal{F}}=s_{\mathcal{F}}.
\end{equation}
\end{theorem}
To prove our main result we verify the following theorem and proposition:

\begin{theorem}\label{cs74}
Let $\mathcal{F}$ be a regular CPLIFS for which the generated self-similar IFS (defined in Section \ref{ct16}) satisfies the Exponential Separation Condition (ESC) (see Definition \ref{cv91}). Then
\begin{equation}\label{cs73}
  \dim_{\rm H} \Lambda^{\mathcal{F}}=
  \dim_{\rm B} \Lambda^{\mathcal{F}}=s_{\mathcal{F}}.
\end{equation}
\end{theorem}
To prove this, we use the multi-parameter Hochman Theorem
\cite[Theorem 1.10]{Hochman_2015}
 and a very recent result of Jordan and Rapaport \cite[Theorem 1.1]{jordan2020dimension}.

\begin{proposition}[Main Proposition]\label{ct46}
A $\dim_{\rm P}$-typical small CPLIFS is regular.
\end{proposition}

Together with \cite[Theorem 3.2]{falconer1997techniques}, this proposition implies the following corollary.
\begin{corollary}\label{cr32}
For a $\dim_{\rm P}$-typical small CPLIFS $\mathcal{F}$ let $s:=\dim_{\rm H}\Lambda^{\mathcal{F}}$ be the Hausdorff dimension of its attractor. Further, assume that all functions in $\mathcal{F}$ are injective.
Then we have
\begin{equation}\label{cr31}
s=\dim_{\rm H} \Lambda^{\mathcal{F}} = \dim_{\rm B} \Lambda^{\mathcal{F}}.
\end{equation}
Further, $\mathcal{H}^s(\Lambda^{\mathcal{F}})<\infty$, where $\mathcal{H}^s$ denotes the $s$-dimensional Hausdorff measure.
\end{corollary}

\section{Self-similar and self-similar graph-directed IFSs on $\mathbb{R}$} \label{cv13}
Here we summarize the relevant results from the theory of self-similar IFSs.
Let $\mathcal{F}=\left\{f_k\right\}_{k=1}^{M}$ be an IFS
 on $\mathbb{R}$ and $\Lambda$ be the attractor of $\mathcal{F}$.
 We denote the symbolic space by $\Sigma=\{1,\ldots,M\}^\N$, and its elements by $\ii=(i_1,i_2,\ldots)\in\Sigma$. We write $\Sigma^\ast$ for the set of all finite length words.
 Elements of $\Sigma^\ast$ are denoted by $\iiv=i_1\ldots i_n$, and sometimes we obtain them as truncations of infinite length words $\ii|n=i_1\ldots i_n$.
As usual, the natural projection $\Pi:\Sigma\to\Lambda$ is defined by
\begin{equation}\label{eq02}
\Pi(\ii) =\lim\limits_{n\to\infty}
f_{\mathbf{i}|n}(0).
\end{equation}
We write $\sigma$ for the left-shift on $\Sigma$. For the definition of invariance, ergodicity, and entropy (denoted by $h(\mu)$) of the measure $\mu$ we refer the reader to Walters' book \cite{walters2000introduction}.
If the measure $\mu$ is ergodic and the derivatives of all of the mappings $f_i$ are continuous, we define the Lyapunov exponent of $\mu$
by
\begin{equation}\label{cv89}
  \chi(\mu):=\int \log |f'_{i_i}(\Pi(\sigma\mathbf{i}))|d\mu(\mathbf{i}).
\end{equation}
Let $\nu$ be a Borel probability measure on $\mathbb{R}$. The Hausdorff dimension of $\nu$ is defined as
\begin{equation}\label{cv82}
  \dim_{\rm H} \nu:=\inf \left\{\dim_{\rm H} E: \nu(E)>0 \right\}.
\end{equation}

\subsection{Self-similar IFSs on the line}

In the special case when all $f_i$ are similarities (the slopes are constants) $\mathcal{F}$ is called self-similar IFS.
In this case for all $i\in[M]:=\left\{1, \dots ,M\right\}$ the mappings $f_i$, can be presented in the form
\begin{equation}\label{cv88}
  f_i(x)=r_ix+t_i, \mbox{ where  } r_i\in(-1,1)\setminus\left\{0\right\}
\mbox{ and }
t_i\in\mathbb{R}.
\end{equation}
The simplest guess for the dimension of the attractor $\Lambda$ can be expressed in terms of the \texttt{similarity dimension} $\dim_{\rm S} \Lambda$, which is defined as the solution $s$ of the equation $\sum\limits_{K=1}^{M}|r_k|^s=1$. That is
\begin{equation}\label{cv85}
  \dim_{\rm H} \Lambda \leq \min\left\{1,\dim_{\rm S} \Lambda\right\},
\end{equation}
It follows from a theorem of  Hochman's  (see in the next subsection) that
 in some sense typically we have equality above.
For an $\iiv:=(i_1, \dots ,i_n)$ we define
\begin{equation}\label{cv87}
r_{\iiv|_{0}}:=1,\quad
  r_{\iiv}=r_{i_1\ldots i_{n}}= r_{i_1}\cdot\ldots\cdot r_{i_{n}}
\mbox{ and }
t_{\iiv}:=\sum\limits_{k=1}^{n}t_{i_k}r_{\iiv|_{k}}.
\end{equation}
Clearly, we have $f_{\iiv}(x)=r_{\iiv}x+t_{\iiv}$ and
\begin{equation}\label{eq:Pi1}
  \Pi(\ii) =
  \sum\limits_{n=1}^{\infty }
  r_{i_1 \dots i_{n-1}}t_{i_n}
=
\lim\limits_{n\to\infty} t_{\mathbf{i}|_n}.
\end{equation}
If $\mu$ is a measure on $\Sigma$ then we write $\Pi_*\mu$ for the push forward measure of $\mu$.
That is $\Pi_*\mu(E)=\mu(\Pi^{-1}E)$.
Let $\mu$ be a $\sigma$-invariant ergodic measure on $\Sigma$.
Then the Lyapunov exponent is
\begin{equation}\label{cv86}
\chi(\mu)=\sum\limits_{k=1}^{m}\mu[k]\log r_k.
\end{equation}
 For a probability vector
$\mathbf{p}:=(p_1, \dots ,p_M)$ we define the measures:
\begin{equation}\label{cv94}
\mu_{\mathbf{p}}([i_1, \dots ,i_n]):=p_{i_1} \cdots p_{i_n}, \mbox{ and }
\nu_{\mathbf{p}}:=\Pi_*\mu_{\mathbf{p}}.
\end{equation}
We say that $\nu_{\mathbf{p}}$ is a \texttt{self-similar measure}.
The simplest guess for the Hausdorff dimension of a self-similar measure
$\nu_{\mathbf{p}}$
is
\begin{equation}\label{cv81}
  \dim_{\rm S}\nu_{\mathbf{p}}:=\frac{h(\nu_{\mathbf{p}})}
  {\chi(\nu_{\mathbf{p}})}
  =
  \frac{\sum\limits_{k=1}^{M}p_k\log p_k}{\sum\limits_{k=1}^{M}p_k\log r_k}.
\end{equation}
$\dim_{\rm H} \nu_{\mathbf{p}} \leq  \dim_{\rm S}\nu_{\mathbf{p}}$ and a theorem of Hochman states that typically we have equality. We say that $\dim_{\rm S} \nu_{\mathbf{p}}$ is the\texttt{ similarity dimension of the measure} $\nu_{\mathbf{p}}$.

Hochman \cite{hochman2014self} introduced the notion of exponential separation for self-similar IFSs. To state it, first we need to define the distance of two similarity mappings
$g_1(x)=\varrho_1x+\tau_1$ and $g_2(x)=\varrho_2x+\tau_2$, $\varrho_1,\varrho_2\in (-1,1)\setminus \left\{0\right\}$,
 on
$\mathbb{R}$. Namely,
\begin{equation}\label{cv92}
  \mathrm{dist}\left(g_1,g_2\right):=
  \left\{
    \begin{array}{ll}
      |\tau_1-\tau_2|, & \hbox{if $\varrho_1=\varrho_2$;} \\
      \infty , & \hbox{otherwise.}
    \end{array}
  \right.
\end{equation}
\begin{definition}\label{cv91}
Given a self-similar IFS $\mathcal{F}=\left\{f_k(x)\right\}_{k=1}^{m}$ on $\mathbb{R}$.
  We say that $\mathcal{F}$ satisfies the \texttt{Exponential Separation Condition (ESC)} if
  there exists a $c>0 $ and a strictly increasing sequence of natural numbers $\left\{n_\ell \right\}_{\ell =1}^{\infty }$ such that
\begin{equation}\label{cv90}
\mathrm{dist}\left(f_{\iiv},f_{\jjv}\right)  \geq c^{n_{\ell }} \mbox{ for all }
\ell  \mbox{ and for all } \iiv,\jjv\in \left\{1, \dots ,M\right\}^{n_\ell },\ \iiv\ne\jjv.
\end{equation}
\end{definition}
We note that the exponential separation condition always holds when an IFS is parametrized by algebraic parameters \cite{hochman2014self}.

\subsubsection{Some important results about self-similar IFSs}\label{cv24}
Here we present three very important theorems about the dimension theory of self-similar IFSs.
In this subsection $\mathcal{F}=\left\{f_k(x)=r_ix+t_i\right\}_{i=1}^{M}$, and $r_k\in(-1,1)\setminus\left\{0\right\}$ as above, and we use the notation of the previous subsections. First we recall two theorems of Hochman.

\begin{theorem}[Hochman \cite{hochman2014self}]\label{cr59}
Let $\mathcal{F}$ be a self-similar IFS on $\mathbb{R}$ which satisfies
  the Exponential Separation Condition (ESC). Then
  \begin{enumerate}[label={\bf (a)}]
  \item
  For every
  self-similar measure $\nu$ we have
  \begin{equation}\label{cv83}
    \dim_{\rm H}\nu= \min\{1,\dim_{\rm S}\Lambda \}.
  \end{equation}
\item Consequently,
  \begin{equation}\label{cv84}
    \dim_{\rm H} \Lambda=\min\left\{1,\dim_{\rm S} \Lambda\right\}
  \end{equation}
  \end{enumerate}
\end{theorem}
We remark that exact overlaps (the existence of $\iiv\ne\jjv$ such that $f_{\iiv}\eqcirc f_{\jjv}$) may lead to dimension drop. 
On the other hand, it follows from Baker \cite{baker2021iterated} and B{\'a}r{\'a}ny, K{\"a}enm{\"a}ki \cite{barany2021super} that ESC can fail even if there is no exact overlap.

The following theorem of Hochman is about the translation family of a self similar IFS.
To state it we need some further notation.

\begin{definition}\label{cv73}
We consider the translation family $\left\{\mathcal{F}^{\pmb{\tau}}\right\}_{\pmb{\tau}\in\mathbb{R}^M}$
(defined in Definition \ref{cv80}) of a self-similar IFS $\mathcal{F}$.
We denote the attractor  of $\mathcal{F}^{\tau}$ by $\Lambda^{\pmb{\tau}}$.
For  $\iiv,\jjv\in[M]^n$ we write

\begin{equation}\label{cv74}
  \Delta_{\iiv,\jjv}(\pmb{\tau}):=f_{\iiv}^{\pmb{\tau}}(0)-
  f_{\jjv}^{\pmb{\tau}}(0).
\end{equation}
Let us define the exceptional set
\begin{equation}\label{cv76}
  E:=\bigcap\limits_{\varepsilon>0}
  \left(\bigcup\limits_{N=1}^{\infty }
  \bigcap\limits_{n>N}
  \left(
  \bigcup\limits_{\iiv,\jjv\in[M]^n \atop \iiv\ne\jjv}
  \Delta^{-1}_{\iiv,\jjv}(-\varepsilon^n,\varepsilon^n)
  \right)
  \right)
\end{equation}
\end{definition}
One can easily check that the following simple fact holds:
\begin{fact}\label{cv72}
Using the notation of Definition \ref{cv73}, we have
\begin{equation}\label{cv71}
  \left\{\pmb{\tau}\in\mathbb{R}^M:
  \mathcal{F}^{\pmb{\tau}} \mbox{ does not satisfy the ESC }
  \right\} \subset E.
\end{equation}
\end{fact}

The following theorem is  a Corollary of
\cite[Theorem 1.10]{Hochman_2015}
\begin{theorem}[Multi parameter Hochman Theorem 
  \cite{Hochman_2015}]\label{cv78}
 Using the notation of Definition \ref{cv73}, we have
 \begin{equation}\label{cv70}
   \dim_{\rm P}E \leq M-1.
 \end{equation}
 Consequently,
 \begin{equation}\label{cv69}
   \dim_{\rm P} \left\{\pmb{\tau}\in\mathbb{R}^M:
  \mathcal{F}^{\pmb{\tau}} \mbox{ does not satisfy the ESC }
  \right\} \leq M-1.
 \end{equation}
\end{theorem}

The consequences of the following theorem will be very important in this paper.
\begin{theorem}[Jordan, Rapaport\cite{jordan2020dimension}]\label{cv68}
  Let $\mathcal{F}=\left\{f_k\right\}_{k=1}^{m}$ be a self-similar IFS on $\mathbb{R}$ which satisfies the ESC. Moreover, let $\mu$ be an invariant ergodic probability measure on $\Sigma=[M]^{\mathbb{N}}$. Then
  \begin{equation}\label{cv67}
    \dim_{\rm H} \Pi_*\mu=\min\left\{1,\frac{h(\mu)}{\chi(\mu)}\right\}.
  \end{equation}
\end{theorem}

In the Appendix in Corollary \ref{cs11} we prove that with the help of this result we can extend part (b) of Theorem \ref{cr59} to self-similar graph-directed iterated function systems.

\subsection{Self-similar graph-directed IFSs on the line}\label{ssec:22}
We present here the most important notations and results related to self-similar Graph-Directed Iterated function Systems (GDIFS).
In this subsection we follow  the book \cite{falconer1997techniques} and the papers \cite{mauldin1988hausdorff} and \cite{keane2003dimension}.
The major difference is that in the first two references the authors assume separation in between the graph-directed sets.
In \cite{keane2003dimension} no separation is assumed, and we follow that line.

To define the graph-directed iterated function systems we need a directed graph $\mathcal{G}=\left(\mathcal{V,E}\right)$.
We label the vertices of this graph with the numbers $\lbrace 1,2,...,q\rbrace$, where $\vert\mathcal{V}\vert =q$. This $\mathcal{G}$ graph is not assumed to be simple, it might have multiple edges between the same vertices, or even loops. For an edge $e=(i,j)\in\mathcal{E}$ we write $s(e):=i$ for the source and $t(e):=j$ for the target of $e$.
Denote with $\mathcal{E}_{i,j}$ the set of directed edges from vertex $i$ to vertex $j$, and write $\mathcal{E}_{i,j}^k$ for the set of length $k$ directed paths between $i$ and $j$. Similarly, we write $\mathcal{E}^n$ for the set of all directed paths of length $n$ in the graph.
We assume that $\mathcal{G}$
is strongly connected. That is for every $i,j\in\mathcal{V}$ there is a directed path in $\mathcal{G}$ from $i$ to $j$.
\par For all edge $e\in \mathcal{E}$ given a contracting similarity mapping $F_e:\mathbb{R} \rightarrow \mathbb{R}$. The contraction ratio is denoted by $r_e \in (-1,1)\setminus \{ 0\}$. Let $e_1 \dots e_n$ be a path in $\mathcal{G}$.
Then we write
$F_{e_1 \dots e_n}:=F_{e_1}\circ\cdots\circ F_{e_n}$.
It follows from the proof of \cite[Theorem 1.1]{mauldin1988hausdorff}
that
 there exists a unique family of non-empty compact sets $\Lambda_1,...,\Lambda_q$ labeled by the elements of $\mathcal{V}$, for which
\begin{equation}\label{cv66}
\Lambda_i=\bigcup\limits_{j=1}^q
\bigcup\limits_{e\in\mathcal{E}_{i,j}}F_e(\Lambda_j),\quad i=1, \dots ,q.
\end{equation}
We call the sets $\lbrace \Lambda_1,...,\Lambda_q\rbrace$ \texttt{graph-directed sets}, and we say that $\Lambda:=\bigcup\limits_{i=1}^{q}\Lambda_i$ is \texttt{the attractor of the
self-similar graph-directed IFS} $\mathcal{F}=\{ F_e\}_{e\in\mathcal{E}}$.
We abbreviate it to \texttt{self-similar GDIFS}.

By iterating \eqref{cv66} we get
\begin{equation}
\Lambda_i=\bigcup_{j=1}^q \bigcup_{(e_1,...e_k)\in \mathcal{E}_{i,j}^k} F_{e_1 \dots e_k} (\Lambda_j).
\label{eq:Graph_union_iter}
\end{equation}

To get the most natural guess for the dimension of $\Lambda$, for every $s \geq 0$,
 we define a $q\times q$ matrix  with the following entries
\begin{equation}\label{ct50}
C^{(s)}=(c^{(s)}(i,j))_{i,j=1}^{q} \mbox{ and }
c^{(s)}(i,j)= \left\{ \begin{array}{ll}
0, & \hbox{if $\mathcal{E}_{i,j}= \emptyset $;} \\
\sum\limits_{e\in \mathcal{E}_{i,j}} |r_{e}|^{s}, & \hbox{otherwise.}
\end{array}
\right.
\end{equation}
The spectral radius of $C^{(s)}$ is denoted by $\varrho(C^{(s)})$. Mauldin and Williams \cite[Theorem 2]{mauldin1988hausdorff} proved that the function
$s\mapsto \varrho(C^{(s)})$ is strictly decreasing, continuous, greater than $1$ at $s=0$, and less than $1$ if $s$ is large enough. 

\begin{definition}\label{cv65}
For the self-similar GDIFS $\mathcal{F}=\left\{f_e\right\}_{e\in \mathcal{E}}$ we write 
$\alpha=\alpha(\mathcal{F})$ for the unique number satisfying
\begin{equation}\label{def:alpha}
\varrho (C^{(\alpha )})=1.
\end{equation}
\end{definition}

The relation of $\alpha$ to the dimension of the attractor is given by the following theorem. Both parts appeared in \cite{mauldin1988hausdorff}
apart from the box dimension part, which is from \cite{falconer1997techniques}.
\begin{theorem}\label{cs14}
Let $\mathcal{F}=\left\{F_e\right\}_{e\in E}$ be a self-similar GDIFS as above.
In particular, the graph $\mathcal{G}=(\mathcal{V},E)$ is strongly connected and let $\Lambda$ be the attractor.
  \begin{enumerate}[label={\bf (a)}]
  \item $\dim_{\rm H} \Lambda \leq \alpha$.
  \item Let $I_k$ be the convex hull of $\Lambda_k$ for all $k\in\mathcal{V}$. If the intervals $\left\{I_k\right\}_{k=1}^{q}$ are pairwise disjoint then
  $\dim_{\rm H} \Lambda=\dim_{\rm B}\Lambda=\alpha $. Moreover,
  $0<\mathcal{H}^ {\alpha}(\Lambda)<\infty $.
  \end{enumerate}
\end{theorem}

\section{The preparation for the proof of the Main Theorem}\label{cs64}
In this rather technical section we will complete the following steps:
 \begin{enumerate}[label={\bf (1)}]
\item First we recall the definition of the generated self-similar IFS $\mathcal{S}_{\mathcal{F}}$ for a CPLIFS
 $\mathcal{F}$.
\item We assume that $\mathcal{F}$ is regular and we consider a suitable sub-system
  $\mathscr{S}_{\mathcal{F}}$ (called associated IFS) of the $N$-th iterate $\mathcal{S}_{\mathcal{F}}^N$ of the generated self-similar IFS $\mathcal{S}_{\mathcal{F}}$ for the appropriate $N$.
\item We represent $\Lambda^{\mathcal{F}}$
as the attractor of a graph directed self-similar IFS whose functions
are elements of $\mathscr{S}_{\mathcal{F}}$.
\item Then we define an appropriate invariant and ergodic measure (actually a Markov measure) for the self-similar IFS $\mathscr{S}_{\mathcal{F}}$, whose support is the attractor of the previously mentioned grapgh directed self-similar IFS which coincides with $ \Lambda_ {\mathcal{F}}$ as we mentioned above.
\end{enumerate}
Later we will apply the Jordan Rapaport theorem (Theorem \ref{cv68}) for this measure to get the dimension of $\Lambda^{\mathcal{F}}$.

\subsection{The generated self-similar IFS}\label{ct16}
Let $\mathcal{F}=\left\{f_k\right\}_{k=1}^{m}\in\mathrm{CPLIFS}_{\pmb{\ell }}$ for a type $\pmb{\ell }=(l(1), \dots ,l(m))$ with $L=\sum\limits_{k=1}^{m}l(k)$.
As we already said in Subsection \ref{cr55}, the \texttt{generated self-similar IFS} $S_{\mathcal{F}}$  consists of those similarity mappings on $\mathbb{R}$ whose graph coincide with the graph of $f_k$ for some $k\in[m]$ on some interval of linearity $J_{k,i}$, $i\in l(k)+1$ of $f_k$.
Hence, it is natural to parameterize the generated self-similar IFS by the $l(k)+1$ intervals of linearity $J_{k,1}, \dots ,J_{k,l(k)+1}$
of the mapping $f_k$ for all $k\in[m]$.
 That is
  \begin{equation}\label{cv34}
    \mathcal{S}:=
    \mathcal{S}_{\mathcal{F}}=
    \left\{S_{k,i}(x)
    =
    \rho_{k,i} \cdot x+t_{k,i}
    \right\}_{k\in[m],i\in[l(k)+1]}, \quad
    S_{k,i}|_{J_{k,i}}\equiv f_k|_{J_{k,i}}.
  \end{equation}
  We organize the translation parts $\left\{t_{k,i}\right\}_{k\in[m],i\in[l(k)+1]}$ of the mappings of $\mathcal{S}_{\mathcal{F}}$
 into a vector
    \begin{equation}\label{cs66}
  \mathfrak{t}=\mathfrak{t}(\mathfrak{b},\pmb{\tau},\pmb{\rho}):=
  (t_{1,1}, \dots ,t_{1,l(1)+1},
   \dots ,
   t_{m,1}, \dots ,t_{m,l(m)+1})\in\mathbb{R}^{L+m}.
  \end{equation}

  If $b_{k,i}>0$ for all $i\in[l(k)]$ then the formula for $t_k$ is simple:
\begin{equation}\label{cv32}
  t_{k,1}=\tau_k \mbox{ and }
  t_{k,i}= \tau_k+
  \sum\limits_{p=1}^{i}b_{k,i}\left(\rho_{k,i}-\rho_{k,i}\right).
\end{equation}
   Using this and an appropriate conjugation in the general case when
   some of $b_{k,i} \leq 0$ we obtain by  simple calculation that the following fact holds:
   \begin{fact}\label{cv30}
    For any fixed $\pmb{\rho}\in\mathfrak{R}^{\pmb{\ell }}$,
consider the mapping $\Phi_{\pmb{\rho}}:\mathfrak{B}^{\pmb{\ell }}\times\mathbb{R}^m
\to\mathbb{R}^{L+m}$ defined by
\begin{equation}\label{cv29}
  \Phi_{\pmb{\rho}}(\mathfrak{b},\pmb{\tau}):=
  \mathfrak{t},
  \end{equation}
  where $\mathfrak{t}\in\mathbb{R}^{L+m}$ was defined in \eqref{cs66}.
Then $\Phi_{\pmb{\rho}}$ is a non-singular affine transformation. Hence
 $\Phi_{\pmb{\rho}}$ and its inverse $\Phi_{\pmb{\rho}}^{-1}$ preserve Hausdorff and packing dimensions, and also preserve the sets of zero measure with respect to $\mathcal{L}^{L+m}$.
\end{fact}

To get a more effective method of labeling the $f_k\in\mathcal{F}$ functions' intervals of linearity, and in this way the functions in $S_{\mathcal{F}}$, we introduce
\begin{equation}\label{cv03}
\mathcal{A}:=\left\{(k,i):k\in[m] \mbox{ and }i\in[l(k)+1]\right\}.
\end{equation}
Moreover, we write 
$\mathcal{A}^N:=\underbrace{\mathcal{A}\times \cdots\times \mathcal{A}}_N$ for $N\in\mathbb{N}\cup\left\{\infty\right\}$.
That is 
$\mathcal{S}_{\mathcal{F}}=\left\{S_{a}\right\}_{a\in\mathcal{A}}$. Note that $\#\mathcal{A}=L+m$.
Recall that the slope $\rho_a$ of $S_a$ was defined in \eqref{cs68}, and we define $t_a\in\mathbb{R}$ such that
\begin{equation}\label{cs98}
  S_{a}(x)=\rho_a x+t_a. 
\end{equation}

\subsection{Associated self-similar and graph-directed self-similar IFS for a regular CPLIFS}\label{cs63}

Let $\mathcal{F}=\left\{f_k\right\}_{k=1}^{m}$ be a regular CPLIFS and let $N=N(\mathcal{F})\in\mathbb{N}$ be the smallest $n$ such that for all $\mathbf{u}=(u_1, \dots ,u_n)\in[m]^n$ the cylinder interval $I_{\mathbf{u}}^{\mathcal{F}}$ does not contain any breaking points.
  Then we say that \texttt{$\mathcal{F}$ is regular of order $N$}.
The collection of regular CPLIFS of type $\pmb{\ell }$ and order $N$
is denoted by $\mathrm{CPLIFS}_{\pmb{\ell },N}$.

Given an  $\mathcal{F}\in\mathrm{CPLIFS}_{\pmb{\ell },N}$.
  Let $\mathcal{G}_{\mathcal{F}}$ be the following directed full graph
 \begin{equation}
 \label{cr09}
 \mathcal{G}_{\mathcal{F}}:=(\mathcal{V},\mathcal{E})
 \mbox{ with }
 \mathcal{V}:=[m]^N \mbox{ and }
 \mathcal{E}:=\left\{ (\mathbf{v},\mathbf{u}):\mathbf{v},\mathbf{u}\in\mathcal{V} \right\}
 \end{equation} 
  Recall from Section \ref{ssec:22} that we denote the source and the target of an edge $e\in\mathcal{E}$ by
$s(e)$ and $t(e)$ respectively. Further, for a $k\in\mathbb{N}$ and $\mathbf{v},\mathbf{u}\in\mathcal{V}$ we write
$\mathcal{E}_{\mathbf{v},\mathbf{u}}^{k}$
for the set of length $k$ directed paths $(e_1, \dots ,e_k)$ such that
$s(e_1)=\mathbf{v}$, $t(e_k)=\mathbf{u}$ and
$t(e_i)=s(e_{i+1})$ for $i\in[k-1]$.

This section is organized in two steps:
\begin{enumerate}
[label={\bf (a)}]
  \item First we associate a self-similar IFS $\mathscr{S}_{\mathcal{F}}$ to every $\mathcal{F}\in\mathrm{CPLIFS}_{\pmb{\ell },N}$
 which is the relevant subsystem of $\mathcal{S}_{\mathcal{F}}^N$.
 \item Then,  we use the directed graph $\mathcal{G}_{\mathcal{F}}=\left( \mathcal{V},\mathcal{E} \right)$ to construct a 
graph-directed self-similar IFS 
$\left\{ F_e \right\}_{e\in\mathcal{E}}$ which consists of the functions of $\mathscr{S}_{\mathcal{F}}$. Moreover,
the attractor $\Lambda ^{\mathcal{G}_{\mathcal{F}}}$ of this self-similar GDIFS coincides with $\Lambda ^{\mathcal{F}}$.
\end{enumerate}

 \bigskip

\textbf{The construction of $\mathscr{S}_{\mathcal{F}}$}.
Let $e=(\mathbf{v},\mathbf{u})\in\mathcal{E}$.
For a  $p\in [N]$ we consider
$f_{\sigma^p\mathbf{v}}I_{\mathbf{u}}$, where $\sigma$
is the left shift. That is $$
 f_{\sigma^p\mathbf{v}}I_{\mathbf{u}}   =\left\{
  \begin{array}{ll}
f_{v_{p+1} \dots v_N}I_{\mathbf{u}}, & \hbox{if $p\in [N-1]$;} \\
I_{\mathbf{u}}    , & \hbox{ if $p=N$.}
  \end{array}
\right.
$$
Clearly, $f_{\sigma^p\mathbf{v}}I_{\mathbf{u}}$
is contained in the $N$-cylinder
$I_{v_{p+1} \dots v_N u_1 \dots u_p}$.  That is $f_{\sigma^p\mathbf{v}}I_{\mathbf{u}}$ contains no breaking points of any functions from $\mathcal{F}$.
Namely, $f_{\sigma^p\mathbf{v}}I_{\mathbf{u}}$ is the subset of a linearity interval of $f_{v_p}$ for every $p\in [N]$.

In particular, there exists a unique $i(e,p)\in [l(v_p)+1]$
such that
\begin{equation}\label{ct21}
 f_{\sigma^p\mathbf{v}}I_{\mathbf{u}} \subset J_{v_p,i(e,p)},
 \mbox{ for any } p\in [N]
\end{equation}

Now we define a mapping $\psi=\psi_{\mathcal{F}}:\mathcal{E}\to \mathcal{A}^N$
\begin{equation}\label{ct25}
\psi(e)=\mathbf{a}=(a_1, \dots ,a_N)\: \mbox{, where }
 a_p:=\left(v_p,i(e,p)\right) .
\end{equation}
Let $e=(\mathbf{v},\mathbf{u}), e'=(\mathbf{v}',\mathbf{u}')\in\mathcal{E}$. It is immediate from the construction that 
\begin{equation}
\label{cr23}
\psi(e)=\psi(e') \Longrightarrow
\mathbf{v}=s(e)=s(e')=\mathbf{v}'.
\end{equation}
We define $\mathscr{A}$ as the image of $\mathcal{E}$ under $\psi$
$$\mathscr{A}:=\mathscr{A}_{\mathcal{F}}
  :=
  \left\{
\mathbf{a}\in\mathcal{A}^N:
\exists e\in\mathcal{E},
  \mathbf{a}=\psi(e)
  \right\}.
  $$ 
Hence, for an $\mathbf{a}\in\mathscr{A}$ it makes sense to write 
\begin{equation}\label{ct12}
  \psi_{1}^{-1}(\mathbf{a}):=\mathbf{v}\quad  \mbox{ and } \quad \mathbf{u}\in\psi_{2}^{-1}(\mathbf{a}),\quad \mbox{ if }
\psi(e)=\mathbf{a} \mbox{ and } e=(\mathbf{v},\mathbf{u}).
\end{equation}
Put
\begin{equation}\label{ct55}
  S_{\mathbf{a}}:=
  S_{a_1}\circ\cdots\circ S_{a_N} \mbox{ for }
 \psi(e)=\mathbf{a}=(a_1, \dots ,a_N) .
\end{equation}
By \eqref{cv34} for any $p\in [N]$ and $a_p=(v_p,i(e,p))$
\begin{eqnarray}\label{cr11}
  S_{a_p}|_{J_{v_p,i(e,p)}}\equiv S_{v_p,i(e,p)}|_{J_{v_p,i(e,p)}}\equiv f_{v_p}|_{J_{v_p,i(e,p)}} .
\end{eqnarray}
Then, using \eqref{cr11} and \eqref{ct21} together it is clear that
\begin{equation}\label{ct23}
 S_{\mathbf{a}}  |_{I_{\mathbf{u}}}
  \equiv
  f_{\mathbf{v}}|_{I _{\mathbf{u}}} ,\mbox{ for } 
  \mathbf{a}=\psi (e) \mbox{ with } e=(\mathbf{v},\mathbf{u})\in\mathcal{E}.
\end{equation}

With the help of alphabet $\mathscr{A}$ 
we define the \texttt{associated self-similar IFS}
\begin{equation}\label{ct51}
\mathscr{S}:=\mathscr{S}_{\mathcal{F}}:=
  \left\{
  S_{\mathbf{a}}
  \right\}_{\mathbf{a}\in\mathscr{A}}.
\end{equation}

The mapping $\psi:\mathcal{E}\to \mathscr{A}$, which we use to associate  the elements of $\mathscr{A}$ to edges, is onto but not necessarily $1-1$.

\bigskip

\textbf{The construction of the associated graph-directed self-similar IFS}
For an $e=(\mathbf{v},\mathbf{u})\in \mathcal{E}$
we consider $S_{\mathbf{a}}$ for $\mathbf{a}=(a_1, \dots ,a_N)=\psi(e)$ and define the mapping $F_e:I_{\mathbf{u}}\to I_{\mathbf{v}}$ by
\begin{equation}\label{ct17}
  F_e(x):=S_{\mathbf{a}}|_{I_{\mathbf{u}}}=f_{\mathbf{v}}|_{I_{\mathbf{u}}}.
\end{equation}
Moreover, for $\mathbf{e}=(e_1,\dots ,e_k)\in \mathcal{E}^k_{\mathbf{v}^1,\mathbf{v}^{k+1}}$ with $e_i=(\mathbf{v}^i,\mathbf{v}^{i+1}),\: i\in[k]$ we have
\begin{equation}\label{cr08}
    \forall x\in I_{\mathbf{v}^{k+1}}:\quad 
    F_{\mathbf{e}}(x):=F_{e_1}\circ\dots\circ F_{e_k}(x)=
    f_{\mathbf{v}^1}\circ\dots\circ f_{\mathbf{v}^k}(x).
\end{equation}
Obviously, $F_e$ is a similarity mapping restricted to $I_{\mathbf{u}}$ with contraction ratio
\begin{equation}\label{cs94}
  \rho_{\mathbf{v},\mathbf{u}}:=\rho_{e}:=
  \rho_{\mathbf{a}}:= S'_{\mathbf{a}} \equiv
  \rho_{a_1}\cdots\rho_{a_N}.
\end{equation}
Recall that we defined the directed full graph in \eqref{cr09}.
We call $\mathcal{F}^{\mathcal{G}}:=\left\{F_e\right\}_{e\in\mathcal{E}}$ the \texttt{associated self-similar graph-directed IFS}.
The symbolic space of $\mathcal{F}^{\mathcal{G}}$ is the set of infinite paths in the full graph $\mathcal{G}_{\mathcal{F}}$ that we denote by $\mathcal{E}_{\infty}$
\begin{eqnarray}\label{ct14}
   \nonumber\mathcal{E}_\infty    &:=& \left\{\mathfrak{p}:=  (e_1,e_2, \dots ):
  t(e_i)=s(e_{i+1}), \ e_i\in\mathcal{E} \mbox{ for all } i\in\mathbb{N}
  \right\} \\
     &=& \left\{\mathfrak{p}=(
  \underbrace{(\mathbf{v}^1,\mathbf{v}^2)}_{e_1},
  \underbrace{(\mathbf{v}^2,\mathbf{v}^3)}_{e_2},
  \underbrace{(\mathbf{v}^3,\mathbf{v}^4)}_{e_3}, \dots
  )
  :
  \mathbf{v}^k \in\mathcal{V},\quad \forall k\in\mathbb{N}
  \right\}.
  \end{eqnarray}

The attractor of $\mathcal{F}^{\mathcal{G}}$ is
\begin{equation}\label{ct19}
  \Lambda^{\mathcal{F}^{\mathcal{G}}}
  :=
  \bigcup_{\mathbf{v}\in\mathcal{V}}
  \bigcap_{k=1}^{\infty }
  \bigcup_{\mathbf{u}\in\mathcal{V}}
  \bigcup_{(e_1 \dots e_k)\in\mathcal{E}_{\mathbf{v},\mathbf{u}}^{k}}
  F_{e_1 \dots e_k}I_{\mathbf{u}}.
\end{equation}

Then by \eqref{cr08} we have
\begin{equation}\label{ct18}
  \Lambda^{\mathcal{F}}= \Lambda^{\mathcal{F}^{\mathcal{G}}}.
\end{equation}

Thus $\Lambda^{\mathcal{F}}$ can be represented as a graph-directed attractor. In the next section we show how $\Lambda^{\mathcal{F}}$ relates to the associated self-similar IFS $\mathscr{S}_{\mathcal{F}}$.

\subsection{Constructing a subshift in $\mathscr{A}^{\mathbb{N}}$}\label{cr33}
Now we show that we can consider $\Lambda^{\mathcal{F}}$ also as the projection of a subshift $\mathscr{A}^{\mathbb{N}}_{\rm Good}\subset \mathscr{A}^{\mathbb{N}}$ (defined below)  by $\Pi_{\mathscr{S}_{\mathcal{F}}}$, the natural projection corresponding to  $\mathscr{S}_{\mathcal{F}}$.
In order to construct $\mathscr{A}^{\mathbb{N}}_{\rm Good}$ we define two bijections: $\vartheta: \mathcal{V}^{\mathbb{N}}\to\mathcal{E}_{\infty}$ and $\Psi:\mathcal{E}_{\infty}\to\mathscr{A}^{\mathbb{N}}_{\rm Good}$.

Our first bijection $\vartheta$ is a very simple one:
\begin{equation}\label{ct09}
  \vartheta(\mathbf{v}^1,\mathbf{v}^2,\mathbf{v}^3,\mathbf{v}^4, \dots
):=
\left(
(\mathbf{v}^1,\mathbf{v}^2),
(\mathbf{v}^2,\mathbf{v}^3),
(\mathbf{v}^3,\mathbf{v}^4),
 \dots
\right).
\end{equation}
To define our second bijection, first
recall the definition of $\psi:\mathcal{E}\to \mathscr{A}$ from
\eqref{ct25}. Now we apply this componentwise to define the mapping $\Psi$ on $\mathcal{E}_\infty$
 by
\begin{equation}\label{ct13}
 \Psi(e_1,e_2, \dots )  =
\left(\mathbf{a}^1,\mathbf{a}^2, \dots \right)\in\mathscr{A}^{\mathbb{N}},\quad
\mathbf{a}^k: =\psi(e_k),\quad  k\in\mathbb{N}.
\end{equation}
Let $$\mathscr{A}_{\mathrm{Good}}^{\mathbb{N}}:=\Psi(\mathcal{E}_\infty ).$$
We claim that 
 $\Psi:\mathcal{E}_{\infty}\to\mathscr{A}^{\mathbb{N}}_{\rm Good} $ is a bijection. 
By definition $\psi(e_i)=\mathbf{a}^{i}$ for $i=1,\dots  ,k$.
By \eqref{cr23}, this  determines 
$e_1,\dots  ,e_{k-1}$ uniquely. Since this holds 
for all $k\in\mathbb{N}$ we obtain that $\Psi $ is $1-1$.
Clearly $\mathscr{A}_{\mathrm{Good}}^{\mathbb{N}}$ is compact and forward invariant. That is
\begin{equation}\label{cs99}
\sigma(\mathscr{A}_{\mathrm{Good}}^{\mathbb{N}}) \subset
 \mathscr{A}_{\mathrm{Good}}^{\mathbb{N}}\subset \mathscr{A}^{\mathbb{N}}.
\end{equation}
That is $\mathscr{A}_{\mathrm{Good}}^{\mathbb{N}}$ is a subshift 
in $\mathcal{A}^{\mathbb{N}}$.

As we defined earlier, let $\Pi_{\mathcal{F}^N}:\mathcal{V}^{\mathbb{N}}\to\mathbb{R}$ and  $\Pi_{\mathscr{S}_{\mathcal{F}}}:\mathscr{A}^{\mathbb{N}}\to
\mathbb{R}$ be  the natural projections corresponding to the  $N$-th iterate system of the CPLIFS $\mathcal{F}$ and the self-similar IFS $\mathscr{S}_{\mathcal{F}}$ respectively.
(For the definition of the natural projection of a general IFS see Section \ref{cv13}.)

Consider the following two diagrams:

\begin{equation}\label{ct06}
  \xymatrix{ \mathcal{V}^{\mathbb{N}} \ar[r]^\sigma  \ar[d]_\vartheta   &
                  \mathcal{V}^{\mathbb{N}}  \ar[d]^{\vartheta }\\ %
\mathcal{E}_\infty  \ar[r]_{\sigma} \ar[d]_\Psi  & \mathcal{E}_\infty  \ar[d]^\Psi \\%
\mathscr{A}_{\mathrm{Good}}^{\mathbb{N}}  \ar[r]_\sigma & \mathscr{A}_{\mathrm{Good}}^{\mathbb{N}} }
\quad
\xymatrix{\mathcal{V}^{\mathbb{N}} \ar[r]^{\Psi \circ \vartheta} \ar[rd]_{\Pi_{\mathcal{F}^N}}  & %
                 \mathscr{A}_{\mathrm{Good}}^{\mathbb{N}}  \ar[d]^{\Pi_{\mathscr{S}_{\mathcal{F}}} }\\ %
&\Lambda^{\mathcal{F}} }%
\end{equation}

It is obvious that the first diagram is commutative.
\begin{fact}\label{ct08}
The second diagram in \eqref{ct06} is commutative.
\end{fact}

\begin{proof}
  Write $\Pi_{\mathcal{F}^{\mathcal{G}}}$ for the natural projection defined by the associated self-similar GDIFS $\mathcal{F}^{\mathcal{G}}$. 
  Since $\Lambda^{\mathcal{F}}=\Lambda^{\mathcal{F}^{\mathcal{G}}}$
  we may dissect the diagram in the following way:
  \begin{equation}\label{cr10}
    \xymatrix{\mathcal{V}^{\mathbb{N}} \ar[r]^{\vartheta} \ar[rd]_{\Pi_{\mathcal{F}^N}}  & %
                 \mathcal{E}_{\infty}  \ar[d]^{\Pi_{\mathcal{F}^{\mathcal{G}}} }\\ %
&\Lambda^{\mathcal{F}} }
\quad
\xymatrix{\mathcal{E}_{\infty} \ar[r]^{\Psi } \ar[d]_{\Pi_{\mathcal{F}^{\mathcal{G}}}} & %
                 \mathscr{A}_{\mathrm{Good}}^{\mathbb{N}}  \ar[ld]^{\Pi_{\mathscr{S}_{\mathcal{F}}} }\\ %
\Lambda^{\mathcal{F}^{\mathcal{G}}} }%
  \end{equation}
We already showed in \eqref{ct18} that the left diagram in \eqref{cr10} is commutative. Using \eqref{ct17} and the definition of $\mathscr{A}_{\mathrm{Good}}^{\mathbb{N}}$ it is easy to see that the right diagram is also commutative.

\end{proof}

In particular we obtained that
\begin{equation}\label{ct01}
\Lambda^{\mathcal{F}}=\Lambda^{\mathcal{F}^N}
=
 \Pi_{\mathcal{F}^N}(\mathcal{V}^{\mathbb{N}})
=
 \Pi_{\mathscr{S}_{\mathcal{F}}}
\left(
 \mathscr{A}_{\mathrm{Good}}^{\mathbb{N}}
\right).
\end{equation}

\subsection{An ergodic measure on $\Lambda^{\mathscr{S}_{\mathcal{F}}}$ supported by $\Lambda^{\mathcal{F}}$}\label{cs92}
Our aim is to define an appropriate invariant ergodic measure $\mathfrak{m}$  on $\mathscr{A}^{\mathbb{N}}$ which is supported by $\mathscr{A}^{\mathbb{N}}_{\mathrm{Good}}$.
Then we will take its push-forward measure by $\Pi_{\mathscr{S}_{\mathcal{F}}}$. This way, according to \eqref{ct01}, the push-forward measure $\Pi_{\mathscr{S}_{\mathcal{F}}*}\mathfrak{m}$ is supported by $\Lambda^{\mathcal{F}}$. 

For every $\beta \geq 0$ we define the $m^N\times m^N$ matrix $C^{(\beta)}$ just like we did in Subsection \ref{ssec:22}.
First we order the elements of $\mathcal{V}$ according to lexicographical order. This will be the order of the rows and columns of $C^{(\beta)}$. Then set
\begin{equation}\label{cs96}
C^{(\beta)}:=
(|\rho_{\mathbf{v},\mathbf{u}}|^{\beta})_
{(\mathbf{v},\mathbf{u})\in\mathcal{V}\times \mathcal{V}},
\end{equation}
where $\rho_{\mathbf{v},\mathbf{u}}$ was defined in \eqref{cs94}.
According to Definition \ref{cv65}, the number $\alpha =\alpha (\mathcal{F})$ is uniquely defined by $\varrho\left( C^{(\alpha)}\right)=1$.

Since $C^{(\alpha)}$ is an irreducible matrix, both the left and the right
eigenvectors
$\underline{\pmb{\mathfrak{u}}}
=(\mathfrak{u}_{\mathbf{v}})_{\mathbf{v}\in\mathcal{V}}$,
$\underline{\pmb{\upnu}}=(\upnu_{\mathbf{v}})_{\mathbf{v}\in\mathcal{V}}$
corresponding to  eigenvalue $1$ can be choosen to have all positive components.
That is
\begin{equation}\label{cs89}
   \sum\limits_{\mathbf{v}\in\mathcal{V}}
\mathfrak{u}_{\mathbf{v}} \cdot |\rho_{\mathbf{v},\mathbf{u}}|^{\alpha}
=
\mathfrak{u}_{\mathbf{u}},\quad
  \sum\limits_{\mathbf{u}\in\mathcal{V}}
|\rho_{\mathbf{v},\mathbf{u}}|^{\alpha}
\upnu_{\mathbf{u}}=
\upnu_{\mathbf{v}},\quad
\mathfrak{u}_{\mathbf{v}},\upnu_{\mathbf{v}}>0
\mbox{ for all }
\mathbf{v}\in\mathcal{V}.
\end{equation}
We normalize them in such a way that
\begin{equation}\label{cs90}
  \sum\limits_{\mathbf{v}\in\mathcal{V}}
\upnu_{\mathbf{v}}=1,\quad
 \sum\limits_{\mathbf{v}\in\mathcal{V}}
\mathfrak{u}_{\mathbf{v}} \cdot
\upnu_{\mathbf{v}}
=1.
\end{equation}
Now we define the stochastic matrix $P=(p_{\mathbf{v},\mathbf{u}})
_{\mathbf{v},\mathbf{u}\in\mathcal{V}}$
and its stationary distribution $\mathbf{p}=(p_{\mathbf{v}})_{\mathbf{v}\in\mathcal{V}}$, which corresponds to the matrix $C^{(\alpha)}$.
That is
\begin{equation}\label{cs88}
  p_{\mathbf{v},\mathbf{u}}:=
\frac{|\rho_{\mathbf{v},\mathbf{u}}|^\alpha\upnu_{\mathbf{u}}}
{\upnu_{\mathbf{v}}},\qquad
p_{\mathbf{v}}:=(\upnu_{\mathbf{v}} \cdot \mathfrak{u}_{\mathbf{v}}),\quad
\mathbf{v},\mathbf{u}\in\mathcal{V}.
\end{equation}
Clearly,
\begin{equation}\label{cs87}
 \mathbf{p}^T \cdot P =\mathbf{p}^T \mbox{ that is }
\sum\limits_{\mathbf{v}}
\upnu_{\mathbf{v}}  \mathfrak{u}_{\mathbf{v}}
 \cdot
\frac{|\rho_{\mathbf{v},\mathbf{u}}|^\alpha\upnu_{\mathbf{u}}}
{\upnu_{\mathbf{v}}}
=
\upnu_{\mathbf{u}}\mathfrak{u}_{\mathbf{u}}.
\end{equation}
\subsubsection{Auxiliary measures}
Now we consider the one-sided Markov shift on $\mathcal{V}^{\mathbb{N}}$ (see \cite[p. 22]{walters2000introduction})
corresponding to $(\mathbf{p},P)$. This gives us the ergodic Borel measure
$\mu=\mu_{\mathcal{F}}$ on
$\mathcal{V}^{\mathbb{N}}$ defined on the $n$-cylinders
$\left[
\mathbf{v}^1, \dots ,\mathbf{v}^n
\right] \subset \mathcal{V}^n$ by
\begin{equation}\label{cs86}
  \mu\left(\left[
\mathbf{v}^1, \dots ,\mathbf{v}^n
\right]\right):=
p_{\mathbf{v}^1} \cdot p_{\mathbf{v}^1,\mathbf{v}^2}
p_{\mathbf{v}^2,\mathbf{v}^3}\cdots
p_{\mathbf{v}^{n-1},\mathbf{v}^n}.
\end{equation}
Then this extends to an ergodic measure on $\mathcal{V}^{\mathbb{N}}$ (see \cite[Theorem 1.19]{walters2000introduction}).
Using that $\Psi\circ\vartheta:\mathcal{V}^{\mathbb{N}}\to  \mathscr{A}^{\mathbb{N}}_{\mathrm{Good}}$ is a homeomorphism,
and the first diagram in \eqref{ct06} is commutative,
we get that the measure
\begin{equation}\label{cs85}
  \nu:=
\left(\Psi\circ\vartheta\right)_*(\mu)
\end{equation}
is an invariant and ergodic measure on $(\mathscr{A}^{\mathbb{N}}_{\mathrm{Good}},\sigma)$. We have pointed out that 
$\mathscr{A}^{\mathbb{N}}_{\mathrm{Good}}$ is subshift in 
$\mathscr{A}^{\mathbb{N}}$. So we can extend $\nu$
from $\mathscr{A}^{\mathbb{N}}_{\mathrm{Good}}$  to
$\mathscr{A}^{\mathbb{N}}$ in a obvious way, such that after extension we still have an ergodic invariant measure. Namely,
we define the measure $\mathfrak{m}$ on $\mathscr{A}^{\mathbb{N}}$
such that for a Borel set $H \subset \mathscr{A}^{\mathbb{N}}$ 
\begin{equation}\label{cs84}
  \mathfrak{m}(H):=\nu(H\cap \mathscr{A}^{\mathbb{N}}_{\mathrm{Good}}).
\end{equation}
Using that $\mathscr{A}^{\mathbb{N}}_{\mathrm{Good}}$ is compact we obtain that the support $\mathrm{spt}(\mathfrak{m})=
\mathscr{A}^{\mathbb{N}}_{\mathrm{Good}}$. Moreover, it follows from \eqref{cs99} and \cite[Theorem 1.6]{walters2000introduction}
that $\mathfrak{m}$ is an ergodic measure. The invariance of $\mathfrak{m}$ is obvious from the definition.

\subsubsection{The entropy and Lyapunov exponent of $\mathfrak{m}$}\label{cs83}
First we estimate the measure of an $n$-cylinder for an arbitrary
$n\in\mathbb{N}$.
For an  $\mathfrak{a}:=(\mathbf{a}^1,\mathbf{a}^2, \dots ,\mathbf{a}^n, \dots )\in\mathscr{A}^{\mathbb{N}}$
we write
 $
\mathfrak{a}|_n:=
(\mathbf{a}^1,\mathbf{a}^2, \dots ,\mathbf{a}^n )
$.
Then
$$
[\mathfrak{a}|_n]=
\left\{
\widehat{\mathfrak{a}}=(\widehat{\mathbf{a}}^1,
\widehat{\mathbf{a}}^2, \dots
)\in\mathscr{A}^{\mathbb{N}}:
\widehat{\mathbf{a}}^1=\mathbf{a}^1, \dots ,
\widehat{\mathbf{a}}^n=\mathbf{a}^n
\right\}.
$$

For every $n$ and $\mathbf{v}^{n+1}\in\Psi_{2}^{-1}(\mathbf{a}^n)$ we have
\begin{equation}\label{cs79}
  \left[
  \mathbf{v}^1, \dots ,\mathbf{v}^n,\mathbf{v}^{n+1}
  \right] \subset
\left(\Psi\circ\vartheta\right)^{-1} \left( [\mathfrak{a}|_n]\right)
 \subset    \left[
  \mathbf{v}^1, \dots ,\mathbf{v}^n
  \right]
\end{equation}
Hence,
\begin{equation}\label{cs78}
  \mu\left(\left[
  \mathbf{v}^1, \dots ,\mathbf{v}^n,\mathbf{v}^{n+1}
  \right]\right)
   \leq
  \nu\left(\left[
  \mathbf{a}^1, \dots ,\mathbf{a}^n
  \right]\right)
   \leq
   \mu\left(\left[
  \mathbf{v}^1, \dots ,\mathbf{v}^n
  \right]\right)
\end{equation}
By substituting the formulas given in \eqref{cs88} for $p_{\mathbf{v},\mathbf{u}}$ and $p_{\mathbf{v}}$ into the formula \eqref{cs86} we obtain that there exist $C_1,C_2>0$ such that
\begin{equation}\label{cs09}
  C_1 \leq
  \frac{\nu\left(\left[
  \mathbf{a}^1, \dots ,\mathbf{a}^n
  \right]\right)}{|\rho_{\mathbf{a}^1, \dots ,\mathbf{a}^n}|^{\alpha}}
   \leq C_2.
\end{equation}

By definition (see \eqref{cv89}) the Lyapunov exponent of $\mathfrak{m}$
is
\begin{equation}\label{cs82}
  \chi_{\mathfrak{m}}:
  =
  - \sum\limits_{\mathbf{a}\in\mathscr{A}}
\mathfrak{m}([\mathbf{a}]) \cdot
\log |\rho_{\mathbf{a}}|,
\end{equation}
where we defined $\rho_{\mathbf{a}}$ in \eqref{cs94}.
Then by the ergodicity of $\mathfrak{m}$ and by the Shannon-McMillan-Breiman Theorem we have
\begin{equation}\label{cs81}
  h(\mathfrak{m})=
  -\lim\limits_{n\to\infty}
  \frac{1}{n}
  \log\mathfrak{m}
  (
  [\mathfrak{a}|_n]
  )
  \mbox{ for $\mathfrak{m}$-a.a. }\mathfrak{a}\in\mathscr{A}^{\mathbb{N}}.
\end{equation}
Moreover,
\begin{equation}\label{cs80}
   \chi_{\mathfrak{m}}:
  =
  -\lim\limits_{n\to\infty}
  \frac{1}{n}\log
  |S'_{\mathfrak{a}|_n}|
  = -\lim\limits_{n\to\infty}
  \frac{1}{n}\log
  |\rho_{\mathfrak{a}|_n}|
  \mbox{ for $\mathfrak{m}$-a.a. }\mathfrak{a}\in\mathscr{A}^{\mathbb{N}}.
\end{equation}
Putting together \eqref{cs80},\eqref{cs81} and \eqref{cs78} we get
\begin{equation}\label{cs77}
  \frac{h(\mathfrak{m})}{\chi_{\mathfrak{m}}}=\alpha.
\end{equation}

\section{The proof of Theorems \ref{cs74}, \ref{ct48} assuming the Main Proposition}\label{cr60}

This section contains the proofs of our main results, under the assumption that the Main Proposition holds. First we prove some facts and a lemma.

\begin{lemma}\label{cs76}
Let $\mathcal{F}$ be a regular CPLIFS. Then we have
  \begin{equation}\label{cs75}
 \alpha(\mathcal{F}) =  s_{\mathcal{F}}.
  \end{equation}
\end{lemma}
To prove this lemma we will use the following facts.

\begin{definition}\label{cr95}
Let $\mathcal{F}$ be a CPLIFS. We say that $\mathcal{F}$ satisfies the \texttt{Bounded Distortion Property (BDP)} if there exist $0<C_1<C_2$ such that
\begin{equation}\label{cr94}
\forall n,\; \forall \vert \mathbf{j}\vert =n,\; \forall x,y\in I :\quad C_1\leq\frac{|f_{\mathbf{j}}^{\prime}(x)|}{|f_{\mathbf{j}}^{'}(y)|}\leq C_2.
\end{equation}
\end{definition}
It is not true in general that a CPLIFS satisfies the BDP. For example, if a function's fixed point coincides with one of its breaking points, then this condition trivially fails. However, in the regular case the BDP does hold.

\begin{lemma}\label{cr93}
Let $\mathcal{F}$ be a regular CPLIFS of order $N$. Then $\mathcal{F}$ satisfies the bounded distortion property.
\end{lemma}

\begin{proof}\label{cr92}
As there are finitely many words of length at most $N$, the statement trivially holds for $n<N$.
Let $\mathbf{j}=j_1\dots j_n$ with $n>N$. Let $x,y\in I$ be arbitrary numbers.
First we investigate the derivative of $f_{\mathbf{j}}$ at a given $x\in I$.
\begin{equation}\label{cr91}
f_{j_1,\dots ,j_n}^{'}(x)=f_{j_1}^{'}(f_{\sigma \mathbf{j}}(x))\cdots
f_{j_{n-N}}^{'}(f_{\sigma^{n-N} \mathbf{j}}(x))
f_{\sigma^{n-N} \mathbf{j}}^{'}(x).
\end{equation}
We said that $N$ is the order of our CPLIFS. It means that the functions of $\mathcal{F}$ have constant slopes over all the $N$ cylinders. Hence $f_k^{'}(\widehat{x})=f_k^{'}(\widehat{y})$ if $\widehat{x},\widehat{y}$ are elements of the same cylinder interval of level $N$.

For $1\leq l\leq n-N$ the words $\sigma^l \mathbf{j}$ has length at least $N$. It implies that $f_{\sigma^l \mathbf{j}}(I)$ is a subset of a cylinder interval of level $N$. Thus $f_{\sigma^l \mathbf{j}}(x)$ and $f_{\sigma^l \mathbf{j}}(y)$ are contained in the same cylinder interval of level $N$ for $1\leq l\leq n-N$.

We conclude that for $1\leq l\leq n-N$
\begin{equation}\label{cr90}
f_{k}^{'}(f_{\sigma^l \mathbf{j}}(x))=f_{k}^{'}(f_{\sigma^l \mathbf{j}}(y)),
\end{equation}
for any $k\in [m]$.

After substituting \eqref{cr90} into \eqref{cr91} we obtain that
\begin{equation}\label{cr89}
C_1:=\frac{\rho_{\min}^N}{\rho_{\max}} \leq
\frac{|f_{j_1,\dots ,j_n}^{'}(x)|}{|f_{j_1,\dots ,j_n}^{'}(y)|}=
\frac{|f_{\sigma^{n-N} \mathbf{j}}^{'}(x)|}{|f_{\sigma^{n-N} \mathbf{j}}^{'}(x)|}
\leq \frac{\rho_{\max}}{\rho_{\min}^N}=:C_2,
\end{equation}
where $\rho_{min}$ and $\rho_{max}$ were defined in \eqref{cr52}.
\end{proof}

If we assume Proposition\ref{ct46}, we can easily proof Corollary \ref{cr31} with the help of this Fact, and the following theorem.

\begin{theorem}[Theorem 3.2 from \cite{falconer1997techniques}] \label{cr05}
    Let $E$ be a non-empty compact subset of $\mathbb{R}^n$ and let $a>0$ and $r_0>0$. Suppose that for every closed ball $B$ with center in $E$ and radius $r<r_0$ there is a mapping $: E\to E\cap B$ satisfying
    \begin{equation}\label{cr06}
        ar\vert x-y\vert \leq \vert g(x)-g(y)\vert 
        \mbox{ for all } x,y\in E.
    \end{equation} 
    Then, writing $s=\dim_H E$, we have that $\mathcal{H}^s(E)\leq 
    4^s a^{-s}<\infty$ and $\underline{\dim}_B(E) = \overline{\dim}_B(E)=s$.
\end{theorem}

\begin{proof}[Proof of Corollary \ref{cr32} assuming Proposition \ref{ct46}]
According to Proposition \ref{ct46}, it is enough to prove our statement for a small regular CPLIFS $\mathcal{F}=\{ f_i\}_{i=1}^m$.

We proceed by applying Theorem \ref{cr05} to $\Lambda^{\mathcal{F}}$. Without loss of generality we assume that $|\Lambda^{\mathcal{F}}| =1$.
Let $r_0:=1$ and $a=\frac{\rho_{min}}{2C_2}$, where $C_2=\frac{\rho_{\max}}{\rho^N_{\min}}$.
Fix an arbitrary $0<r<r_0$ and $x\in\Lambda^{\mathcal{F}}$, and consider the interval $D= [x-r,x+r]$. We choose $(i_1,\dots ,i_n)\in [m]^n$ such that $x\in \Lambda_{i_1,\dots ,i_n}\subset D$ but $\Lambda_{i_1,\dots ,i_{n-1}}\not\subset D$. Then $\frac{r}{2}<|\Lambda_{i_1,\dots ,i_{n-1}}|$.

Let $I^{\mathcal{F}}$ be the interval we defined in \eqref{cr22}. There exists $x_1\in I^{\mathcal{F}}$ for which
\begin{equation}\label{cr30}
\frac{r}{2}<|\Lambda_{i_1,\dots ,i_{n-1}}|\leq
|f_{i_1,\dots ,i_{n-1}}^{'}(x_1)| .
\end{equation}

According to Lemma \ref{cr93} $\mathcal{F}$ satisfies the bounded distortion property. Namely, \eqref{cr89} implies that for every $x_2\in I^{\mathcal{F}}$ we have
\begin{equation}\label{cr29}
|f_{i_1,\dots ,i_{n-1}}^{'}(x_1)|\leq C_2|f_{i_1,\dots ,i_{n-1}}^{'}(f_n(x_2))|\leq
C_2 \underbrace{|f_{i_1,\dots ,i_{n-1}}^{'}(f_{i_n}(x_2))| \cdot |f_{i_n}^{'}(x_2)|}_{|f^{'}_{i_1,\dots ,i_n}(x_2)|}\frac{1}{\rho_{\min}},
\end{equation}
where in the last inequality we used that $\rho_{\min}$ is the smallest slope in the system in absolute value.
Together \eqref{cr30} and \eqref{cr29} gives that for all $x_2\in I^{\mathcal{F}}$
\begin{equation}\label{cr27}
|f^{'}_{i_1,\dots ,i_n}(x_2)|>\frac{\rho_{\min}}{2C_2}r=a\cdot r.
\end{equation}
All the functions in $\mathcal{F}$ are continuous and piecewise linear, hence for any $y,z\in\Lambda^{\mathcal{F}},\: y<z$ there exists a $\xi\in (y,z)$ such that
\begin{equation}\label{cr28}
|f_{i_1,\dots ,i_n}(z)-f_{i_1,\dots ,i_n}(y)| \geq
|f_{i_1,\dots ,i_n}^{'}(\xi )|\cdot (z-y)\geq a\cdot r(z-y),
\end{equation}
where we used that \eqref{cr27} applies for any element of $I^{\mathcal{F}}$.

Since $f_{i_1,\dots ,i_n}\Lambda^{\mathcal{F}} \subset \Lambda^{\mathcal{F}}\cap D$, all conditions of Theorem \ref{cr05} hold which completes the proof of our theorem.

\end{proof}

\begin{fact}\label{cr88}
Let $\mathcal{F}$ be a regular CPLIFS.
Then
\begin{equation}\label{cr87}
\Phi(s)=\lim\limits_{n\to\infty} \frac{1}{n}\log\sum\limits_{i_1, \dots ,i_n} |I_{i_1 \dots i_n}|^s.
\end{equation}
That is we can write $\lim $ instead of $\limsup$ in \eqref{cr64}.
\end{fact}

\begin{proof}
Since $\mathcal{F}$ is regular, it satisfies the BDP according to Lemma \ref{cr93}. It follows that for every $n$ there is a suitably big $k\geq n-N$ and constants $D_1=D_1(N),D_2=D_2(N)$ such that
\begin{equation}\label{cr86}
D_1 \left\vert I_{i_1,\dots ,i_{kN}}\right\vert \leq
\left\vert I_{i_1,\dots ,i_{n}}\right\vert \leq
D_2 \left\vert I_{i_1,\dots ,i_{kN}}\right\vert .
\end{equation}

Using this we have that
\begin{align}\label{cr97}
\Phi (s) &=\limsup_{n\rightarrow\infty} \frac{1}{n} \log \sum_{i_1,\dots ,i_{n}} \left\vert I_{i_1,\dots ,i_{n}}\right\vert^s \\
&=\limsup_{k\rightarrow\infty} \frac{1}{kN} \log \sum_{i_1,\dots ,i_{kN}} \left\vert I_{i_1,\dots ,i_{kN}}\right\vert^s \nonumber \\
&=\limsup_{k\rightarrow\infty} \frac{1}{kN} \log \sum_{i_1,\dots ,i_{kN}} \left\vert f_{i_1,\dots i_{(k-1)N}}\left(I_{i_{(k-1)N+1},\dots ,i_{kN}}\right) \right\vert^s \nonumber \\
&=\limsup_{k\rightarrow\infty} \frac{1}{kN} \log \sum_{\mathbf{v}^1, \dots ,
\mathbf{v}^k\in\mathcal{V}} \left\vert
F_{\mathbf{v}^1,\mathbf{v}^2}\circ \cdots\circ
F_{\mathbf{v}^{k-1},\mathbf{v}^k}
\left(I_{\mathbf{v}^k}\right) \right\vert^s \nonumber \\
&=\limsup_{k\rightarrow\infty} \frac{1}{kN} \log \sum_{\mathbf{v}^1, \dots ,
\mathbf{v}^k\in\mathcal{V}}
|\rho_{\mathbf{v}^1,\mathbf{v}^2}|^s\cdots
|\rho_{\mathbf{v}^{k-1},\mathbf{v}^k}|^s
\left\vert \left(I_{\mathbf{v}^k}\right) \right\vert^s \nonumber \\
&=\limsup_{k\rightarrow\infty} \frac{1}{kN} \log \sum_{\mathbf{v}^1, \dots ,
\mathbf{v}^k\in\mathcal{V}}
|\rho_{\mathbf{v}^1,\mathbf{v}^2}|^s\cdots
|\rho_{\mathbf{v}^{k-1},\mathbf{v}^k}|^s , \nonumber
\end{align}
where in the last equality we used that the length of any cylinder interval of level $N$ can be easily bounded by suitable constants.

The series $c_k:=\log \sum_{\mathbf{v}^1, \dots ,\mathbf{v}^k\in , \mathcal{V}} |\rho_{\mathbf{v}^1,\mathbf{v}^2}|^s\cdots
|\rho_{\mathbf{v}^{k-1},\mathbf{v}^k}|^s$ is subadditive, and hence by Fekete's lemma the limit $\lim_{k\rightarrow\infty}\frac{c_k}{k}$ exists.
Applying this fact to \eqref{cr97} yields
\begin{equation}\label{cr83}
\Phi (s)=\lim_{k\rightarrow\infty} \frac{1}{kN} \log
\sum_{\mathbf{v}^1, \dots ,\mathbf{v}^k\in\mathcal{V}}
|\rho_{\mathbf{v}^1,\mathbf{v}^2}|^s\cdots
|\rho_{\mathbf{v}^{k-1},\mathbf{v}^k}|^s.
\end{equation}

Together \eqref{cr83} and \eqref{cr97} imply that
\begin{equation}\label{cr82}
\Phi (s) =\lim_{n\rightarrow\infty} \frac{1}{n} \log \sum_{i_1,\dots ,i_{n}} \left\vert I_{i_1,\dots ,i_{n}}\right\vert^s .
\end{equation}

\end{proof}

\begin{proof}[The proof of Lemma \ref{cs76}]
We define
\begin{align} \label{cr99}
\overline{Q}&:=\max_{i_1,\dots ,i_N} \left\vert I_{i_1,\dots ,i_N}\right\vert ,\quad
\underline{Q}:=\min_{i_1,\dots ,i_N} \left\vert I_{i_1,\dots ,i_N}\right\vert ,\\
\overline{q}&:=\max_{\mathbf{u}, \mathbf{v}\in [m]^N} \mathfrak{u}_{\mathbf{u}}\upnu_{\mathbf{v}} ,\quad
\underline{q}:=\min_{\mathbf{u}, \mathbf{v}\in [m]^N} \mathfrak{u}_{\mathbf{u}}\upnu_{\mathbf{v}} .
\end{align}

By Fact \ref{cr88} and its proof we have
\begin{equation}\label{cr81}
\Phi (s)=\lim_{k\rightarrow\infty} \frac{1}{kN} \log \sum_{\mathbf{v}^1, \dots ,\mathbf{v}^k\in\mathcal{V}}
|\rho_{\mathbf{v}^1,\mathbf{v}^2}|^s\cdots
|\rho_{\mathbf{v}^{k-1},\mathbf{v}^k}|^s
\left\vert \left(I_{\mathbf{v}^k}\right) \right\vert^s.
\end{equation}

Observe that by the definition of measure $\mu$ we have
\begin{equation}\label{cr98}
\underline{q}|\rho_{\mathbf{v}^1,\mathbf{v}^2}\cdots \rho_{\mathbf{v}^{k-1},\mathbf{v}^k}|^{\alpha} \leq
\mu ([\mathbf{v}^1,\dots ,\mathbf{v}^k])
\leq \overline{q}|\rho_{\mathbf{v}^1,\mathbf{v}^2}\cdots \rho_{\mathbf{v}^{k-1},\mathbf{v}^k}|^{\alpha},
\end{equation}

Now we substitute $s=\alpha (=\alpha_{\mathcal{F}})$ into \eqref{cr97}, and use the bounds of \eqref{cr98} and the constants introduced in \eqref{cr99} to obtain
\begin{align}\label{cr96}
\Phi (\alpha )&\leq
\lim_{k\rightarrow\infty} \frac{1}{kN} \log \sum_{\mathbf{v}^1, \dots ,\mathbf{v}^k\in\mathcal{V}}
|\rho_{\mathbf{v}^1,\mathbf{v}^2} \cdots
\rho_{\mathbf{v}^{k-1},\mathbf{v}^k}|^{\alpha}\overline{Q}^{\alpha} \\
&\leq
\lim_{k\rightarrow\infty} \frac{1}{kN}\frac{\overline{Q}^{\alpha}}{\underline{q}} \log \underbrace{
\sum_{\mathbf{v}^1,\dots \mathbf{v}^{k}\in\mathcal{V}} \mu ([\mathbf{v}^1,\dots \mathbf{v}^{k}])}_{=1} =0, \nonumber
\end{align}
since $\mu$ is a probability measure. Note that it only holds for $s=\alpha (\mathcal{F})$, otherwise $\mu$ would fail to be a probability measure.
Similar calculations show that $0$ is also a lower bound.

\begin{equation}\label{cr84}
\Phi (\alpha )\geq
\lim_{k\rightarrow\infty} \frac{1}{kN}\frac{\underline{Q}^{\alpha}}{\overline{q}} \log \underbrace{
\sum_{\mathbf{v}^1,\dots \mathbf{v}^{k}\in\mathcal{V}} \mu ([\mathbf{v}^1,\dots \mathbf{v}^{k}])}_{=1} =0. 
\end{equation}

As $s_{\mathcal{F}}$ is the unique zero of $\Phi(s)$, we just proved that $s_{\mathcal{F}}=\alpha$.
\end{proof}

\begin{proof}[Proof of Theorem \ref{cs74}]
  Now we fix an arbitrary regular CPLIFS $\mathcal{F}$ of order $N$, for which the generated self-similar IFS $\mathcal{S}_{\mathcal{F}}$ satisfies ESC.
  Let $\mathscr{S}:=\mathscr{S}_{\mathcal{F}}:=\lbrace S_{\mathbf{a}}\rbrace_{\mathbf{a}\in\mathcal{A}}$ be its associated self-similar system as defined in  Section \ref{cs63}.
  Since the ESC holds for $\mathcal{S}_{\mathcal{F}}$, it also holds for $\mathscr{S}$, as the later one is a subsystem of $\mathcal{S}_{\mathcal{F}}^N$ (see \cite{barany2016dimension}).

Recall that the measure $\mathfrak{m}$ defined in \eqref{cs84} satisfies
\begin{equation}\label{cr70}
\mbox{spt}(\Pi_{\mathscr{S}_{\mathcal{F}}\ast }\mathfrak{m})=\Lambda^{\mathcal{F}}.
\end{equation}
We also know that $h(\mathfrak{m})/\chi (\mathfrak{m})=\alpha$ by \eqref{cs77}. 
Now we can apply the
Jordan Rapaport
 Theorem (Theorem \ref{cv68}) to obtain
\begin{equation}\label{cr69}
\min \lbrace 1,\alpha\rbrace \leq \dim_{\rm H} \Lambda^{\mathcal{F}}.
\end{equation}
Applying Lemma \ref{cs76} yields
\begin{equation}\label{cr68}
\min \lbrace 1,s_{\mathcal{F}}\rbrace \leq \dim_{\rm H} \Lambda^{\mathcal{F}}\leq
\dim_{\rm B} \Lambda^{\mathcal{F}}
\leq \min \lbrace 1,s_{\mathcal{F}}\rbrace ,
\end{equation}
where the last inequality comes from Corollary \ref{cr12}.
\end{proof}

Now we can prove our Main Theorem.
\begin{proof}[Proof of Theorem \ref{ct48} assuming  Proposition \ref{ct46}]
Fix an arbitrary type \newline
  $\pmb{\ell }=(l(1), \dots ,l(m))$ and
  $\pmb{\rho}\in \mathfrak{R}_{\mathrm{small}}^{\pmb{\ell }}$. Let us define the exceptional sets
  \begin{equation}\label{cs62}
E_{\pmb{\rho}}
:=
\left\{(\mathfrak{b},\pmb{\tau})
\in \mathfrak{B}^{\pmb{\ell }}\times \mathbb{R}^m
:
\dim_{\rm H} \Lambda^{(\mathfrak{b},\pmb{\tau},\pmb{\rho})}
\ne
s_{(\mathfrak{b},\pmb{\tau},\pmb{\rho})}
\right\},
  \end{equation}
  \begin{equation}\label{cs61}
E_{\pmb{\rho},\mathrm{ESC}}
:=
\left\{(\mathfrak{b},\pmb{\tau})
\in \mathfrak{B}^{\pmb{\ell }}\times \mathbb{R}^m
:
\mathscr{S}_{\mathcal{F}^  {(\mathfrak{b},\pmb{\tau},\pmb{\rho})}}
\mbox{ does not satisfy ESC}
\right\},
  \end{equation}
  \begin{equation}\label{cs60}
E_{\pmb{\rho},\mathrm{irregular}}
:=
\left\{(\mathfrak{b},\pmb{\tau})
\in \mathfrak{B}^{\pmb{\ell }}\times \mathbb{R}^m
:
\mathcal{F}^{(\mathfrak{b},\pmb{\tau},\pmb{\rho})}
\mbox{ is not regular }
\right\}.
  \end{equation}
Using this terminology Theorem \ref{cs74} asserts that
\begin{equation}\label{cs53}
  E_{\pmb{\rho}} \subset
E_{\pmb{\rho},\mathrm{ESC}}\cup   E_{\pmb{\rho},\mathrm{irregular}}.
\end{equation}
Recall that the mapping $\Phi_{\pmb{\rho}}$ (defined in \eqref{cv29}) associates each $(\mathfrak{b},\pmb{\tau})\in \mathfrak{B}^{\pmb{\ell}}\times\mathbb{R}^m$ to a vector $\mathbf{t}\in\mathbb{R}^{L+m}$ uniquely.
$\Phi_{\pmb{\rho}}(\mathfrak{b},\pmb{\tau})$ is the translation vector of $\mathscr{S}_{\mathcal{F}^{(\mathfrak{b},\pmb{\tau},\pmb{\rho})}}$.
Hence, by Hochmann's theorem (Theorem \ref{cv78})
\begin{equation}\label{cs56}
\dim_{\rm P} \Phi_{\pmb{\rho}}(E_{\pmb{\rho},\mathrm{ESC}}) <L+m.
\end{equation}
According to Fact \ref{cv30}, $\Phi_{\pmb{\rho}}$ preserves Packing dimension, and thus
\begin{equation}\label{cs55}
\dim_{\rm P} E_{\pmb{\rho},\mathrm{ESC}} <L+m.
\end{equation}
It follows from Proposition \ref{ct46} that
\begin{equation}\label{cs59}
\dim_{\rm H} E_{\pmb{\rho},\mathrm{irregular}}<L+m.
\end{equation}
Combining \eqref{cs53} with the last two inequalities completes the proof of the Theorem.
\end{proof}

\section{The proof of the Main Proposition}\label{cr58}

\begin{proof}[Proof of Proposition \ref{ct46}]

 Our aim in this section is to verify Proposition \ref{ct46}. That is we will prove that the attractor $\Lambda^{\mathcal{F}}$ of a $\dim_{\rm P}$-typical small CPLIFS $\mathcal{F}$ does not contain any breaking points. More precisely, we prove that for every $\pmb{\rho} \in\mathfrak{R}_{\mathrm{small}}^{\pmb{\ell }}$
 \begin{equation}\label{cs45}
 \dim_{\rm P}
 \left(
 E_{\pmb{\rho},\mathrm{irregular}}
 \right) \leq L+m-1+s_*,
 \end{equation}
 where $E_{\pmb{\rho},\mathrm{irregular}}$ was defined in \eqref{cs60} and $0<s_*<1$ is defined by
\begin{equation}\label{cs44}
  \sum\limits_{k=1}^{m} \rho_k^{s_*}=1.
\end{equation}

By symmetry, it is enough to prove that
\begin{equation}\label{cs52}
  b_{1,1}\not\in\Lambda^{\mathcal{F}},\quad
  \mbox{ for a $\dim_{\rm P}$-typical small CPLIFS }\mathcal{F}.
\end{equation}
Fix an arbitrary type $\pmb{\ell }=(l(1), \dots ,l(m))$ and a
$\pmb{\rho} \in\mathfrak{R}_{\mathrm{small}}^{\pmb{\ell }}$.
We always assume that $(\mathfrak{b},\pmb{\tau})\in \mathfrak{B}^{\pmb{\ell }}\times \mathbb{R}^m$, and write
  $\Lambda^{(\mathfrak{b},\pmb{\tau},\pmb{\rho})}$
for the  attractor
of the CPLIFS $\mathcal{F}^{(\mathfrak{b},\pmb{\tau},\pmb{\rho})}$.
As usual, we write $L:=\sum\limits_{k=1}^{m}l(k)$.

For arbitrary $U<V$ we write
\begin{equation}\label{cs51}
  E_{U,V}:=
  \left\{
(\mathfrak{b},\pmb{\tau})\in \left(\mathfrak{B}^{\pmb{\ell }}\cap (U,V)^{L}\right)\times
 (U,V)^m
  :
  b_{1,1}\in \Lambda^{(\mathfrak{b},\pmb{\tau},\pmb{\rho})} \subset (U,V)
  \right\}.
\end{equation}
To show that \eqref{cs45} holds, it is enough to prove that
for all $U<V$ we have
\begin{equation}\label{cs50}
  \dim_{\rm P} E_{U,V} \leq L+m-1+s_*.
\end{equation}
From now on we fix  $U<V$, and we write
\begin{equation}\label{cs49}
  \Delta:=
  \left\{
  (\mathfrak{b},\pmb{\tau})
\in \left(
  \mathfrak{B}^{\pmb{\ell }}\cap(U,V)^L
  \right)\times
  (U,V)^m:
   \Lambda^{(\mathfrak{b},\pmb{\tau},\pmb{\rho})} \subset (U,V)
  \right\}.
\end{equation}
Let $D$ be an arbitrary closed cube contained in $\Delta$:
\begin{equation}\label{cs47}
  D=D_1\times D_2, \mbox{ where }
  D_1 \subset \mathbb{R}\mbox{ and } D_2 \subset \mathbb{R}^{L+m-1}.
\end{equation}
To verify \eqref{cs50} we need to prove that
\begin{equation}\label{cs48}
   \dim_{\rm P} (E_{U,V}\cap D) \leq L+m-1+s_*.
\end{equation}
For an  $r>0$ we define a kind of symbolic Moran cover
\begin{equation}\label{cs46}
  \mathcal{M}_r:=
  \left\{
  (k_1, \dots ,k_n)\in\Sigma^*:
  |\rho_{k_1 \dots k_n}| \leq r< |\rho_{k_1 \dots k_{n-1}|}
  \right\}.
\end{equation}
By a standard argument it is easy to see that
\begin{equation}\label{cs43}
  \frac{1}{r^{s_*}}
   \leq \#\mathcal{M}_r  <
    \frac{1}{\rho_{\mathrm{min}}^{s_*}r^{s_*}}.
\end{equation}
Clearly, for all $r>0$ we have
\begin{equation}\label{cs42}
  E_{U,V}\cap D=\bigcup_{(k_1 \dots k_n)\in\mathcal{M}_r}
\widetilde{K}_{k_1, \dots ,k_n}.
\end{equation}
where $\widetilde{K}_{k_1, \dots ,k_n}:=\left\{(\mathfrak{b},\pmb{\tau})\in   D:
b_{1,1}\in \Lambda_{k_1 \dots k_n}^{(\mathfrak{b},\pmb{\tau},\pmb{\rho})}
\right\}.$
Let $\widetilde{N}_{k_1 \dots k_n}(r)$ be the minimal number of $r$-mesh cubes required to cover
$\widetilde{K}_{k_1, \dots ,k_n}$.
Using \eqref{cs43}, it is enough to prove that
there exists a constant $C_1$ which may depend on $D$ but independent of $r$
and $(k_1, \dots ,k_n)\in\mathcal{M}_r$ such that
\begin{equation}\label{cs41}
\widetilde{N}_{k_1 \dots k_n}(r)
 \leq
 \frac{C_1}{r^{L+m-1}},\quad \forall  r>0,\
 (k_1, \dots ,k_n)\in\mathcal{M}_r.
\end{equation}
Namely, putting together this and \eqref{cs46} we get from \eqref{cs42} that
$$\overline{\dim}_{\rm B}\left(E_{U,V}\cap D\right) \leq L+m-1+s_* $$
for all cubes $D \subset \Delta$, which implies \eqref{cs48}. 
Using that $\Delta$ is the countable union of such cubes we obtain \eqref{cs50}.

Without loss of generality we may assume that $U=1$.
Now we fix an $r>0$ which is so small that
for every $(\mathfrak{b},\pmb{\tau})\in D$ and $(k_1, \dots ,k_n)\in\mathcal{M}_r$ we have
 \begin{equation}\label{cr67}
   f_{k_1 \dots k_n}^{(\mathfrak{b},\pmb{\tau},\pmb{\rho})}
   ([0,V+1]) \subset \left(\frac{1}{2},V+\frac{1}{2}\right)
   \mbox{ and }
   |\rho_{k_1 \dots k_n}| <\min\left\{|D_1|,
   \frac{1}{6(V-1)}
   \right\}.
 \end{equation}
We also fix an arbitrary $(k_1, \dots ,k_n)\in\mathcal{M}_r$. Now we verify \eqref{cs41}.
  For the duration of this proof we introduce the following notation
 \begin{equation}\label{cs40}
\pmb{\eta}:=(b_{1,2}, \dots ,b_{1,l(1)}, \dots ,
b_{m,1}, \dots ,b_{m,l(m)},\tau_1, \dots ,\tau_m)\in D_2.
 \end{equation}
That is, instead of
$(\mathfrak{b},\pmb{\tau})$ from now on
 we write $(b_{1,1},\pmb{\eta})\in D=D_1\times D_2 .$
Let $B_k^{(b_{1,1},\pmb{\eta}) }$ be the set of breaking points of the function $f_k^{(b_{1,1},\pmb{\eta},\pmb{\rho})}(\cdot)$ and we write $B^{(b_{1,1},\pmb{\eta}) }_{k_1\dots k_n}$
for the set of breaking points of $f^{(b_{1,1},\pmb{\eta},\pmb{\rho})}_{k_1\dots k_n}(\cdot) $. Then
\begin{equation}\label{cs39}
B^{(b_{1,1},\pmb{\eta}) }_{k_1\dots k_n}=  B_{k_n}^{(b_{1,1},\pmb{\eta}) }\cup\bigcup_{p=2}^{n}f_{k_n}^{-1}\circ\cdots\circ
  f_{k_{p}}^{-1}B_{k_{p-1}}^{(b_{1,1},\pmb{\eta}) }.
\end{equation}
The complement $\left(B^{(b_{1,1},\pmb{\eta}) }_{k_1\dots k_n}\right)^c$ is the union for all
$(i_1, \dots ,i_n)\in[l(k_1)+1]\times\cdots\times[l(k_n)+1]$ of the open intervals
\begin{eqnarray}\label{cv36}
\nonumber
     J_{(k_1,i_1) \dots (k_n, i_n)} \!\! \!\!  \!\! &=& \!\!  \!\! J_{k_n,i_n}\cap S_{k_n,i_n}^{-1}(J_{k_{n-1},i_{n-1}})
\cap \cdots\cap
S_{k_n,i_n}^{-1}\circ\cdots\circ  S_{k_2,i_2}^{-1}(J_{k_1,i_1})\\
    \!\!  \!\!&=& \!\!\!\!
      \left\{x\in    J_{k_n,i_n}  \!\!: \!\!
      S_{k_n,i_n}(x)\in J_{k_{n-1},i_{n-1}}, \dots ,
      S_{(k_2,i_2)\cdots(S_{k_n,i_n})}(x)\in J_{k_{1},i_{1}}
       \right\}.
   \end{eqnarray}
Let $ \overline{J}_{(k_1,i_1) \dots (k_n, i_n)}$ be the closure of
$J_{(k_1,i_1) \dots (k_n, i_n)}$, then by the formulas \eqref{cv34}
and \eqref{cv32}
we obtain that for all $ x\in \overline{J}_{(k_1,i_1) \dots (k_n, i_n)}$ we have
\begin{multline}\label{cs37}
 f_{k_1 \dots k_n}^{(b_{1,1},\pmb{\eta},\pmb{\rho}) }(x)  =
\widetilde{\rho}_n \cdot x\\
+
\sum\limits_{j=1}^{n}
\widetilde{\rho}_{j-1}\underbrace{\left(
\tau_{k_j}+b_{k_j,1}(\rho_{k_j,1}-\rho_{k_j,2})+\cdots+
b_{k_j,i_j}(\rho_{k_j,i_j}-\rho_{k_j,i_j+1})
\right)}_{t_{k_j,i_j}},
\end{multline}
where $\widetilde{\rho}_q:=\prod\limits_{z=1}^{q}\rho_{k_z,i_z}$.
Using this, a simple calculation shows that the following fact holds:
\begin{fact}\label{ct84}

  For a fixed $x$
  and $k_1, \dots ,k_n$ the following hold
   \begin{enumerate}[label={\bf (a)}]
\item  the function
  $f_{k_1 \dots k_n}^{(b_{1,1},\pmb{\eta},\pmb{\rho})}(x)$ is  piecewise linear in all of its variables $\tau_k$ and $b_{p,q}$. So, if we fix all variables but one, then the partial derivative, against the variable which is not fixed, exists at all but finitely many points. 
   \end{enumerate}
  
The next two assertions are meant in the sense that the estimates hold where the partial derivatives exist.

  \begin{enumerate}[resume,label={\bf (a)}]
\item   For every $p\in[m]$ and $q\in[l(p)]$ we have
  \begin{equation}\label{ct83}
  \left|\frac{\partial f_{k_1 \dots k_n}^{(b_{1,1},\pmb{\eta},\pmb{\rho})}(x)}{\partial b_{p,q}}\right|<\gamma:=\sum\limits_{k=1}^{\infty }\rho_{\max}^k<1.
  \end{equation}
\item For every $k\in[m]$
  \begin{equation}\label{ct82}
  \left|\frac{\partial f_{k_1 \dots k_n}^{(b_{1,1},\pmb{\eta},\pmb{\rho})}(x)}{\partial \tau_{k}}\right|<2.
  \end{equation}
  \end{enumerate}
\end{fact}

Clearly,
\begin{equation}\label{cs36}
  \Lambda^{(b_{1,1},\pmb{\eta}) }_{k_1 \dots k_n} \subset I^{(b_{1,1},\pmb{\eta}) }:=
  \left(
  f^{(b_{1,1},\pmb{\eta}) }_{k_1 \dots k_n}(1)-(V-1)\rho_{k_1 \dots k_n},
  f^{(b_{1,1},\pmb{\eta}) }_{k_1 \dots k_n}(1)+(V-1)\rho_{k_1 \dots k_n}
  \right),
\end{equation}
and we also have
\begin{equation}\label{cs34}
\widetilde{K}_{k_1 \dots k_n} \subset
K_{k_1 \dots k_n}:=
\left\{(b_{1,1},\pmb{\eta}) \in D:
b_{1,1}\in  I^{(b_{1,1},\pmb{\eta}) }\right\}.
\end{equation}
Like above, we write $N_{k_1 \dots k_n}(r) $ for the minimal number of $r$-mesh cubes required to cover the set $K_{k_1 \dots k_n}$.
By \eqref{cs34} $\widetilde{N}_{k_1 \dots k_n}(r) \leq N_{k_1 \dots k_n}(r)$.
Hence to verify \eqref{cs41} we only need to prove that
\begin{lemma}\label{cs20}
 There is a constant $C_1$ which is independent of $r$ and $(k_1, \dots ,k_n)\in\mathcal{M}_r$
such that
\begin{equation}\label{cs21}
N_{k_1 \dots k_n}(r)
 \leq
 \frac{C_1}{r^{L+m-1}},\quad \forall  r>0,\
 (k_1, \dots ,k_n)\in\mathcal{M}_r.
\end{equation}
\end{lemma}

 \begin{proof}[Proof of Lemma \ref{cs20}]
For our arbitrary but fixed $r>0$ and $(k_1, \dots ,k_n)\in \mathcal{M}_r$
we define the mapping $G:D\to \left(\frac{1}{3},V+\frac{2}{3}\right)$,
\begin{equation}\label{cs35}
  G(b_{1,1},\pmb{\eta}) :=f^{(b_{1,1},\pmb{\eta}) }_{k_1 \dots k_n}(1)-(V-1)\rho_{k_1 \dots k_n}.
\end{equation}
Observe that \eqref{ct83} gives
\begin{equation}\label{cs24}
  \left|
  \frac{\partial}{\partial b_{1,1}}G(b_{1,1},\pmb{\eta})
  \right|  < \gamma \quad \mbox{ for all } \pmb{\eta}\in D_2.
\end{equation}
Then for $w:=2(V-1)\rho_{k_1 \dots k_n}$ we have
\begin{equation}\label{cs33}
  K_{k_1 \dots k_n}=
\left\{
(b_{1,1},\pmb{\eta}) \in D:
b_{1,1}-w
<G(b_{1,1},\pmb{\eta})
<b_{1,1}
\right\} .
\end{equation}
We introduce the stripe $Z:=\left\{(x,y)\in D_1\times\mathbb{R}:
 x-w<y<x \right\}$ and
for an $\pmb{\eta}\in D_2$ we write
\begin{equation}\label{cs23}
K_{k_1 \dots k_n}^{\pmb{\eta}}:=\left\{b_{1,1}\in D_1:
(b_{1,1},G(b_{1,1},\pmb{\eta})) \in Z
\right\}.
\end{equation}
It follows from \eqref{ct83} that for all $\pmb{\eta}\in D_2$ we have
\begin{equation}\label{cs32}
  |G(b_{1,1},\pmb{\eta}) -G(b_{1,1}+u,\pmb{\eta})  |<\gamma
  |u|, \quad\  b_{1,1}, b_{1,1}+u\in D_1,\pmb{\eta}\in D_2.
\end{equation}
This immediately implies that
\begin{fact}\label{cs31} For every $x$ we have
  \begin{enumerate}[label={\bf (a)}]
  \item $\bigg(G(b_{1,1},\pmb{\eta})  \geq b_{1,1}-x\bigg)\Longrightarrow \bigg(G(b_{1,1}-u,\pmb{\eta}) >
  b_{1,1}-u-x\bigg) ,$ \\
  if  $b_{1,1}, b_{1,1}-u\in D_1$,
  \item $\bigg(G(b_{1,1},\pmb{\eta}) \leq b_{1,1}-x\bigg)
  \Longrightarrow
  \bigg(G(b_{1,1}+u,\pmb{\eta})<b_{1,1}+u-x\bigg) ,$ \\
  if  $b_{1,1}, b_{1,1}+u\in D_1$.
  \end{enumerate}
\end{fact}

Let us write $d_l$ and $d_r$ for the left and right endpoint of $D_1$ respectively. For any set $A\subset\mathbb{R}^{L+m}$ we define a projective mapping
\begin{equation}\label{cr62}
\mathrm{proj}_2 (A):=\left\{ \pmb{\eta}:\: \exists b_{1,1} \mbox{ such that } (b_{1,1},\pmb{\eta})\in A\right\}.
\end{equation}
Fact \ref{cs31} implies that
\begin{equation}\label{cs27}
  \pmb{\eta}\in \mathrm{Proj}_2(K_{k_1 \dots k_n})
  \Longleftrightarrow
  d_l-w<G(d_l,\pmb{\eta})\quad \&\quad  G(d_r,\pmb{\eta})<d_r.
\end{equation}
That is
\begin{equation}\label{cs25}
  \pmb{\eta}\not\in \mathrm{Proj}_2(K_{k_1 \dots k_n})
  \Longleftrightarrow
 G(d_l,\pmb{\eta}) \leq d_l
 \mbox{ or }
 G(d_r,\pmb{\eta}) \geq d_r.
\end{equation}

Assume that $d_l<G(d_l,\pmb{\eta})$ and $G(d_r,\pmb{\eta})<d_r$ holds. Then there is a  fixed point
of the function $G(\cdot,\pmb{\eta})$. Let us denote it by $\mathrm{Fix}(G(\cdot,\pmb{\eta}))$. Now we define the function
$q:D_2\to[d_l,d_r]$ as follows:
\begin{equation}\label{cs26}
q(\pmb{\eta}):=
\left\{
  \begin{array}{ll}
  \mathrm{Fix}(G(\cdot,\pmb{\eta}))  , & \hbox{ if $d_l<G(d_l,\pmb{\eta})$ and $G(d_r,\pmb{\eta})<d_r$;} \\
    d_l, & \hbox{if $G(d_l,\pmb{\eta}) \leq d_l$;} \\
    d_r, & \hbox{if $G(d_r,\pmb{\eta}) \geq d_r$.}
  \end{array}
\right.
\end{equation}

\begin{figure}[h]
  \centering
  \includegraphics[width=13cm]{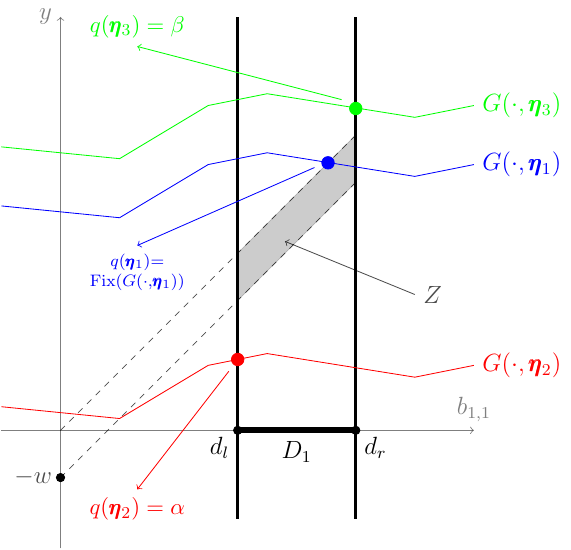}
  \caption{Visualization of the function $q(\pmb{\eta})$ defined in \eqref{cs26}. The continuous piecewise linear map $G(\cdot ,\pmb{\eta})$ is drawn for three different $\pmb{\eta}$ values to show examples for all cases.}
  \label{cs08}
\end{figure}

\begin{fact}\label{cs19}
  The function $q:D_2\to [d_l,d_r]$ has the following properties
  \begin{enumerate}[label={\bf (a)}]
\item  $q(\pmb{\eta})$ is piecewise affine.
\item There is a constant $K_0$ such that all of the  partial derivatives of the function $q(\pmb{\eta})$ are less than $K_0$ in
    absolute value.
  \end{enumerate}
\end{fact}
Part (a) is immediate from the formula \eqref{cs37}. We obtain part (b) from \eqref{ct83} and \eqref{ct82} either by a direct calculation
or by applying the Implicit Function Theorem. Now we put together \eqref{cs24}, \eqref{cs23} and Fact \ref{cs31} to obtain
\begin{equation}\label{cs22}
 K_{k_1 \dots k_n}^{\pmb{\eta}} \subset
 \left[q(\pmb{\eta}),q(\pmb{\eta})+\frac{2(V-1)\rho_{k_1 \dots k_n}}{1-\gamma}\right], \quad \mbox{ for all }\quad \pmb{\eta}\in D_2.
\end{equation}
In this way $K_{k_1 \dots k_n} \subset
  W_{k_1 \dots k_n}$, where
\begin{equation}\label{cs18}
  W_{k_1 \dots k_n}:=\left\{
(b_{1,1},\pmb{\eta}): \pmb{\eta}\in D_2,
b_{1,1}\in
 \left[q(\pmb{\eta}),q(\pmb{\eta})+\frac{2(V-1)\rho_{k_1 \dots k_n}}{1-\gamma}\right]
  \right\}
\end{equation}

Recall that $D=D_1\times D_2$ is a cube and $D_1=(d_l,d_r)$. That is
$D_2$ is a translate of the cube $(0,d_r-d_l)^{L+m-1}$.
Hence we can cover $D_2$ by $\left(\left\lceil
\frac{d_r-d_l}{\rho_{k_1 \dots k_n}}
\right\rceil +1\right)^{L+m-1}\quad$ ($L+m-1$)-dimensional $r$-mesh cubes, since
$(k_1, \dots ,k_n)$ $\in\mathcal{M}_r$ and in this way
\begin{equation}\label{cs17}
  \rho_{\min}r<\rho_{k_1 \dots k_n} \leq r.
\end{equation}
Let $Q$ be one of these $(L+m-1)$-dimensional $r$-mesh cubes.
Then it follows from Fact \ref{cs19} and \eqref{cs22} that the set $\mathrm{Proj}_2^{-1}(Q)\cap W_{k_1 \dots k_n}$
can be covered by $(L+m-1) \cdot K_0 \cdot \frac{2(V-1)}{1-\gamma}\quad$ $(L+m)$-dimensional $r$-mesh cubes. That is we can cover $W_{k_1 \dots k_n}$
(and consequently $K_{k_1 \dots k_n}$) by
$$
\left(\left\lceil
\frac{d_r-d_l}{\rho_{k_1 \dots k_n}}
\right\rceil +1\right)^{L+m-1}
 \cdot
 (L+m-1) \cdot K_0 \cdot \frac{2(V-1)}{1-\gamma}
 \leq
 \frac{C_1}{r^{L+m-1}},
$$
 $r$-mesh cubes with a suitable constant $C_1$ independent of $r$ and $(k_1,\dots ,k_n)\in \mathcal{M}_r$, where in the last step we used \eqref{cs17}.
This completes the proof of the Lemma.
\end{proof}

As we already mentioned, Lemma \ref{cs20} implies that \eqref{cs41} holds. That is, using \eqref{cs43} and \eqref{cs42}, we can cover $E_{U,V}\cap D$ with at most 
$\frac{C_2}{r^{L+m-l+s_*}}\quad$ $(L+m)$-dimensional $r$-mesh cubes. As it holds for all $D\subset\Delta$ and $\Delta$ is the countable union of such cubes, we just obtained \eqref{cs50}. Since $U$ and $V$ were arbitrary, the proof of the Main Proposition follows.

\end{proof}

\section{Appendix}

The purpose of the Appendix is to make explicit an important consequence of Jordan and Rapaport's result
\cite[Theorem 1.1]{jordan2020dimension} related to the dimension of the attractors of graph-directed self-similar systems on $\mathbb{R}$.

Given a self-similar graph-directed IFS $\mathcal{F}$ with a strongly connected graph.
We will construct an ergodic and invariant measure $\mu$ on the symbolic space which is a Markov measure with the following property:
the  entropy of $\mu$ divided by the  Lypunov exponent of $\mu$ is equal to $\alpha (\mathcal{F})$, where $\alpha (\mathcal{F})$ was defined in  Definition \ref{cv65}. That is why $\mu$ can be considered as the natural measure for the 
self-similar graph-directed IFS $\mathcal{F}$.

Using this and the Jordan-Rapaport Theorem (Theorem \ref{cv68}),
we obtain in Corollary \ref{cs11},
that the Hausdorff dimension of the push-forward measure of $\mu$ 
is the minimum of $1$ and $\alpha (\mathcal{F})$ if ESC holds for the 
self-similar IFS $\mathcal{S}$ associated to $\mathcal{F}$.
This can be considered as a generalization of part (b) of Hochman Theorem (Theorem \ref{cr59}) for graph-directed self-similar systems.

\medskip

Throughout the Appendix we use the notation of Section \ref{ssec:22}.
Consider a self-similar graph-directed iterated function system $\mathcal{F}=\{ F_e(x)=r_e x+t_e\}_{e\in\mathcal{E}}$ with directed graph $\mathcal{G}=\{\mathcal{V},\mathcal{E} \}$.
We may identify
\begin{equation}\label{cr43}
  \mathcal{V}=[q]=\left\{1, \dots ,q\right\}, \mbox{ and }
  \mathcal{E}=\left\{e_1, \dots ,e_M\right\}.
\end{equation}
We assume that $\mathcal{G}$ is strongly connected.
Let $\mathcal{E}_{\infty}$ be the set of infinite length directed paths in $\mathcal{G}$.
Moreover, for $l\in [q]$ we introduce
$$\mathcal{E}_{\infty}^{\ell }
:=\left\{ 
\mathbf{e}=(e_1,e_2,\dots  )\in\mathcal{E}_{\infty}:
s(e_1)=\ell 
 \right\}.$$

We identify $\mathcal{E}_{\infty  }$ with $\Sigma _Z:=\left\{(i_1,i_2,\dots  )\in
[M]^{\mathbb{N}}
:
z_{i_k,i_{k+1}}=1,\: \forall k\in\mathbb{N}
\right\}$, where
 $Z=(z_{i,j})_{i,j=1}^m$ is an $M\times M$ matrix such that 
\begin{equation}\label{cr14}
z_{i,j}=
\left\{
\begin{array}{ll}
1
,&
\hbox{if $t(e_i)=s(e_j)$;}
\\
0
,&
\hbox{otherwise.}
\end{array}
\right.
\end{equation} 
Set $\Pi_{\mathcal{F}}(e_1,e_2,\dots  ):=\lim_{n\to\infty  }
F_{e_1 \dots   e_n}(0)$. It is clear that the non-empty compact  sets 
$\left\{ \Pi_{\mathcal{F}}(\mathcal{E}^{\ell }_{\infty  })  \right\}_{\ell \in[q]}$ satisfy \eqref{cv66}. Hence 
\begin{equation}
\label{cr24}
\Lambda _i=\Pi_{\mathcal{F}}(\mathcal{E}^{i }_{\infty  }),
\mbox{ and }
\Lambda=\Pi_{\mathcal{F}}(\mathcal{E}_{\infty  })
\end{equation}

We define the natural pressure function as
\begin{equation}\label{cs13}
\Phi (s)=\lim_{n\rightarrow\infty}\frac{1}{n}\log\sum_{\mathbf{e}\in\mathcal{E}^n} |F_{\mathbf{e}}(\Lambda_{t(e_n)})|^s,
\end{equation}
where the existence of the limit follows from the same standard argument used in the proof of Fact \ref{cr88}. We assumed that $\mathcal{G}$ is strongly connected, which implies that the matrix $C^{(s)}$ we defined in \eqref{ct50} is irreducible.
It is easy to see that by the Perron Frobenius Theorem we have
\begin{equation}\label{cr45}
\Phi (s)=\log \varrho (C^{(s)}),
\end{equation}
where $\varrho$ is the spectral radius. According to \cite[Theorem 2]{mauldin1988hausdorff} $\Phi (s)$ is continuous, strictly decreasing, $\Phi (0)> 0$, and $\lim_{s\rightarrow\infty}\Phi (s)=-\infty$, thus there exists a unique $0<s_{\mathcal{F}}$ for which $\Phi (s_{\mathcal{F}})=0$.
For a self-similar GDIFS with strongly connected graph we call  $s_{\mathcal{F}}$ the natural dimension of the system.
By Definition \ref{cv65} it is clear that
\begin{equation}\label{cr46}
s_{\mathcal{F}} = \alpha.
\end{equation}

\begin{definition}\label{cr35}
  We call $\mathcal{S}=\left\{S_k(x)=r_{e_k}x+t_{e_k}\right\}_{k=1}^{M}$ \texttt{the self-similar IFS associated with the self-similar GDIFS} $\mathcal{F}$. Clearly,
\begin{equation}\label{cr36}
S_{k}|_{\Lambda_{t(e_k)}}\equiv F_{e_k}|_{\Lambda_{t(e_k)}}.
\end{equation}
\end{definition}

Now we show that 
Part (a) of Theorem \ref{cs14} combined with Jordan Rapaport Theorem
 (Theorem \ref{cv68}) implies that under some conditions, part (b) of Hochman's theorem  holds for self-similar graph directed attractors as well, with $s_{\mathcal{F}}$ in the place of $\dim_{\rm S}\Lambda$.

\begin{corollary}\label{cs11}
Let $\mathcal{F}=\left\{F_e\right\}_{e\in \mathcal{E}}$ be a self-similar GDIFS on $\mathbb{R}$ with directed graph $\mathcal{G}=(\mathcal{V},\mathcal{E})$ and attractor $\Lambda$. Further, assume that $\mathcal{G}$ is strongly connected, and that the self-similar IFS $\mathcal{S}$ associated to $\mathcal{F}$ satisfies the ESC. Then
\begin{equation}\label{cs15}
\dim_{\rm H}\Lambda = \min \{ 1,s_{\mathcal{F}}\} .
\end{equation}
\end{corollary}

\begin{proof}\label{cs10}
  Since $\mathcal{G}$ is strongly connected, we can apply Theorem \ref{cs14} to obtain
\begin{equation}\label{cr49}
\dim_{\rm H}\Lambda \leq \min\{\alpha ,1\} \mbox{, since } \Lambda\subset\mathbb{R}.
\end{equation}
In order to prove the opposite inequality, first we define an ergodic invariant measure $\mu$ on $\mathcal{V}^{\mathbb{N}}$.  To do so, we recall the definition of $\alpha$ from \eqref{def:alpha} and using that we consider the irreducible matrix $A=(a_{i,j})_{i,j=1}^{q}:=C^{(\alpha)}$. Let $\mathbf{u}=(u_1, \dots ,u_q)$ and $\mathbf{v}=(v_1, \dots v_q)$
be the left and right eigenvector of $A$ respectively, corresponding to the leading eigenvalue $1$, normalized in a way that
\begin{equation}\label{cr44}
  \sum\limits_{k=1}^{q}u_k v_k=1,\quad
  \sum\limits_{k=1}^{q}v_k=1, \quad
  u_k,v_k>0.
\end{equation}
We define the stochastic matrix
$P=(p(i,j))_{i,j=1}^q$ and its stationary distribution $\mathbf{p}=(p_1, \dots ,p_q)$, related to the matrix $A$. That is
\begin{equation}\label{cr42}
  p(i,j):=\frac{a(i,j)v_j}{v_i},\quad
  p_i:=u_i \cdot v_i, \quad i,j\in[q].
\end{equation}
Clearly, $\mathbf{p}$ is a probability vector and we have
\begin{equation}\label{cr41}
  \mathbf{p}^T \cdot P=\mathbf{p}^T.
\end{equation}
With the help of these we can define the $(\mathbf{p},P)$ Markov measure $\mu$ on $n$-cylinders $[v_1,\dots ,v_n]\in\mathcal{V}^n$ as follows \cite[p. 22]{walters2000introduction}
\begin{equation}\label{cr40}
\mu ([v_1,\dots ,v_n]):= p_{v_1}\cdot p_{v_1,v_2}p_{v_2,v_3}\cdots p_{v_{n-1},v_n}.
\end{equation}
According to \cite[Theorem 1.19]{walters2000introduction} this extends to an ergodic, invariant measure on $\mathcal{V}^{\mathbb{N}}$.  There is a natural bijection between $\mathrm{spt}(\mu)\subset\mathcal{V}^{\mathbb{N}}$ and $\mathcal{E}_{\infty}$, where $\mathrm{spt}(\cdot)$ denotes the support of the measure and
\begin{equation}\label{cr34}
  \mathrm{spt}(\mu)=
  \left\{
  (v_1,v_2, \dots )\in \mathcal{V}^{\mathbb{N}}:
  (v_i,v_{i+1})\in\mathcal{E}, \mbox{ for all } i\in\mathbb{N}
  \right\}.
\end{equation}
For a $(v_1,v_2, \dots )\in\mathrm{spt}(\mu)$
we define
\[
  \vartheta(v_1,v_2,\dots):=((v_1,v_2),(v_2,v_3),\dots)\in\mathcal{E}_{\infty},
\]
where the inclusion holds since
 $\mu$ is only supported on vertex series $(v_1,v_2,\dots)$ $\in\mathcal{V}^{\mathbb{N}}$ for which $(v_i,v_{i+1})\in\mathcal{E}$ for all $i\geq 1$.
Thus the pushforward measure
$\nu:=\vartheta_*\mu$ onto $\mathcal{E}_{\infty}$ is also ergodic and invariant with respect to the left shift $\sigma$. Observe that there is a natural embedding $\Psi:\left(\mathcal{E}_{\infty },\sigma\right)\to \left([M]^{\mathbb{N}},\sigma\right)$ defined as $\Psi(e_{i_1},e_{i_2},\dots  ):=(i_1,i_2,\dots  )$, where we write $\sigma$
also for left shift on $[M]^{\mathbb{N}}$.
With the help of the matrix \eqref{cr14} we already defined 
$\Sigma_Z=\Psi(\mathcal{E}_{\infty})$. 
The push forward measure $\Psi_*\nu$  is an invariant ergodic measure supported on $\Sigma_Z$. 

Let $\Pi_{\mathcal{S}}:[M]^{\mathbb{N}}\to \mathbb{R}$ be the natural projection corresponding to $\mathcal{S}$.
All the symbolic spaces and their relations we introduced so far are summarized on these diagrams: 
\begin{equation}\label{cr25}
  \xymatrix{ \mathrm{spt}(\mu) \ar[r]^\sigma  \ar[d]_\vartheta   &
                  \mathrm{spt}(\mu)  \ar[d]^{\vartheta }\\ %
\mathcal{E}_\infty  \ar[r]_{\sigma} \ar[d]_\Psi  & \mathcal{E}_\infty  \ar[d]^\Psi \\%
\Sigma _Z  \ar[r]_\sigma & \Sigma _Z }
\quad
\xymatrix{\mathcal{E}_{\infty} \ar[r]^{\Psi } \ar[rd]_{\Pi_{\mathcal{F}}}  & %
                 \Sigma _Z \ar[d]^{\Pi_{\mathcal{S}}}\\ %
&\Lambda }%
\end{equation}
It is immediate that both diagrams are commutative.
This and \eqref{cr24} imply that
  \begin{equation}\label{cr50}
  \mathrm{spt}\left( \Pi_{\mathcal{S}\ast}(\Psi_*\nu)\right) =\Lambda.
\end{equation}                
Moreover, easy calculations give us the entropy and the Lyapunov exponent of $\Psi_*\nu$:
\begin{align}\label{cr39}
h(\Psi_*\nu)&=-\lim_{n\to\infty}\frac{1}{n}\log\nu (\mathbf{e}|_n)
=-\alpha\lim_{n\to\infty}\frac{1}{n}\log r_{\mathbf{e}|n} \\
\chi_{\Psi_*\nu}&=-\lim_{n\to\infty}\frac{1}{n}\log r_{\mathbf{e}|n},\nonumber
\end{align}
for $\nu$-almost all $\mathbf{e}=(e_1,e_2,\dots )\in\mathcal{E}_{\infty}$, where $r_{\mathbf{e}|n}=r_{e_1\dots e_n}$ stands for the contraction ratio of $F_{\mathbf{e}|n}$.
Hence we have
\begin{equation}\label{cr38}
\frac{h(\Psi_*\nu)}{\chi_{\Psi_*\nu}}=\alpha.
\end{equation}
 We assumed that the ESC holds for $\mathcal{S}$, and we have seen that $\Psi_*\nu$ is an invariant and ergodic probability measure.  Thus we can apply the Jordan Rapaport Theorem (Theorem \ref{cv68}) to obtain that
\begin{equation}\label{cr37}
\dim_{\rm H}\Lambda \geq \min\{\alpha ,1\},
\end{equation}
by substituting \eqref{cr50} and \eqref{cr38} into the theorem.

Together \eqref{cr37} and \eqref{cr49} yields
\begin{equation}\label{cr47}
\min\{\alpha ,1\}\leq \dim_{\rm H}\Lambda \leq \min\{\alpha ,1\}.
\end{equation}

According to \eqref{cr46} $s_{\mathcal{F}}=\alpha$, hence \eqref{cr47} implies
\[
\dim_{\rm H}\Lambda =\min\{s_{\mathcal{F}} ,1\}.
\]
\end{proof}

\begin{acknowledgement}
  R. Dániel Prokaj and Károly Simon acknowledge support from the grant
OTKA K123782.
\end{acknowledgement}

\bibliographystyle{abbrv}
\bibliography{CPLIFS_bib}

\end{document}